\documentclass[12pt, a4paper]{amsart}

\usepackage[english]{babel}
\usepackage{amsmath}
\usepackage{amssymb}
\usepackage{bm}
\usepackage{amscd} 
\usepackage{microtype} 
\usepackage{verbatim} 
\usepackage{color}  
\usepackage{tikz-cd}\usetikzlibrary{babel} 
\usepackage{enumitem}
\usepackage{mathtools} 
\usepackage{cite}
\usepackage{hyperref} 
\usepackage[initials,msc-links,backrefs]{amsrefs} 
\usepackage{dynkin-diagrams}

\usepackage[OT2, T1]{fontenc}
\title{On the Euler characteristic of \(S\)-arithmetic groups}

 \author[H. Kammeyer]{Holger Kammeyer}
 \author[G. Serafini]{Giada Serafini}
 
 \address{Heinrich Heine University D{\"u}sseldorf, Faculty of Mathematics and Natural Sciences, Mathematical Institute, Germany}
 \email{holger.kammeyer@hhu.de}
 \email{giada.serafini@hhu.de}
 

\theoremstyle{plain}
\newtheorem{theorem}[equation]{Theorem}

\newtheorem{corollary}[equation]{Corollary}
\newtheorem{proposition}[equation]{Proposition}

\theoremstyle{definition}

\newtheorem*{definition*}{Definition}

\newtheorem*{observation*}{Observation}

\providecommand{\ignore}[1]{}

\providecommand{\R}{\mathbb{R}}
\providecommand{\Q}{\mathbb{Q}}
\providecommand{\Z}{\mathbb{Z}}

\providecommand{\C}{\mathbb{C}}

\DeclareMathOperator{\Aut}{Aut}
\DeclareMathOperator{\Ad}{Ad}

\DeclareMathOperator{\Br}{Br}

\newcommand*{\arXiv}[1]{ \href{http://www.arxiv.org/abs/#1}{arXiv:\textbf{#1}}}

\begin{document}

\begin{abstract}
  We show that the sign of the Euler characteristic of an \(S\)-arithmetic subgroup of a simple algebraic group depends on the \(S\)-congruence completion only, except possibly in type \({}^6 D_4\).  Consequently, the sign is a profinite invariant for such \(S\)-arithmetic groups with the congruence subgroup property.  This generalizes previous work of the first author with Kionke--Raimbault--Sauer.
\end{abstract}

\maketitle

\section{Introduction}

The purpose of this article is to prove a stability result for the sign \(\operatorname{sgn} \chi(\Gamma)\) of the Euler characteristic of \(S\)-arithmetic groups, where \(\operatorname{sgn}(a)\) is defined to be \(-1\), \(0\), \(1\) if \(a < 0\), \(a=0\), \(a > 0\), respectively.

\begin{theorem} \label{thm:main-theorem}
  For \(i=1,2\), let \(k_i\) be number fields and let \(S_i\) be finite sets of places of \(k_i\) containing the infinite ones.  Let \(\Gamma_i \le \mathbf{G_i}\) be \mbox{\(S_i\)-arithmetic} subgroups of simply-connected simple non-triality \mbox{\(k_i\)-groups} with commensurable congruence completions.  Then \(\operatorname{sgn} \chi (\Gamma_1) = \operatorname{sgn} \chi (\Gamma_2)\).
\end{theorem}

By \emph{non-triality}, we mean that \(\mathbf{G_i}\) is not of type \({}^6 D_4\).  This assumption can be dropped if \(k_i = \Q\).  We will however discuss in Section~\ref{section:full-triality} why this exception might be essential for certain number fields.

Here is some background and motivation for the theorem.  A group invariant \(\alpha\) is called \emph{profinite} among a class of groups \(\mathcal{C}\) if for any two \(\Gamma_1, \Gamma_2 \in \mathcal{C}\) with isomorphic profinite completions \(\widehat{\Gamma_1} \cong \widehat{\Gamma_2}\), we have \mbox{\(\alpha(\Gamma_1) = \alpha(\Gamma_2)\)}.  While it was observed in~\cite{Bridson-Conder-Reid:fuchsian}*{Corollary~3.3} that the first \emph{\(\ell^2\)-Betti number}~\citelist{\cite{Lueck:l2-invariants} \cite{Kammeyer:l2-invariants}}, denoted \(b^{(2)}_1(\Gamma)\), is profinite among finitely presented residually finite groups, it is shown in~\cite{Kammeyer-Sauer:spinor} that no higher \(\ell^2\)-Betti number \(b^{(2)}_n(\Gamma)\) for \(n \ge 2\) is profinite among \(S\)-arithmetic groups.  The examples of~\cite{Kammeyer-Sauer:spinor} leave however the possibility that \(\operatorname{sgn} \chi (\Gamma)\) is determined by \(\widehat{\Gamma}\) for \(S\)-arithmetic \(\Gamma\).  Since \(\chi(\Gamma) = \sum_{n \ge 0} (-1)^n \,b^{(2)}_n(\Gamma)\), this would imply that some information on higher \(\ell^2\)-cohomology is reflected in \(\widehat{\Gamma}\), even though the individual \(\ell^2\)-Betti numbers are not.  While the main result in~\cite{Kammeyer-et-al:profinite-invariants} gave confirmation for arithmetic groups with the \emph{congruence subgroup property} (CSP), we can now handle the case of \(S\)-arithmetic groups with CSP.

\begin{theorem} \label{thm:sign-profinite}
  For \(i=1,2\), let \(k_i\) be number fields and let \(S_i\) be finite sets of places containing the infinite ones.  Suppose \(\Gamma_i \le \mathbf{G_i}\) are \mbox{\(S_i\)-arithmetic} subgroups of simply-connected simple non-triality \mbox{\(k_i\)-groups} with CSP such that \(\widehat{\Gamma_1}\) and \(\widehat{\Gamma_2}\) are commensurable.  Then \(\operatorname{sgn} \chi (\Gamma_1) = \operatorname{sgn} \chi (\Gamma_2)\). 
\end{theorem}

We should emphasize that the absolute value of the Euler characteristic is not profinite among \(S\)-arithmetic groups~\cite{Kammeyer-et-al:profinite-invariants}*{Theorem~1.2}.  In fact, profiniteness of group invariants among \(S\)-arithmetic groups seems to fail more often than not, so positive results are valuable.  In particular, one can find pairs of profinitely commensurable \(S\)-arithmetic groups illustrating that neither group homology (in degree \(\ge 2\)), nor geometric dimension, nor Kazhdan's property~(\(T\))~\cite{Aka:property-t}, nor Serre's property FA~\cite{Cheetham-West-et-al:property-fa}, nor bounded cohomology~\cite{Echtler-Kammeyer:bounded}, nor the Bohr compactification~\cite{Jaikin-Lubotzky:grothendieck-pairs} are profinite.

Theorem~\ref{thm:sign-profinite} gives the generalization that was asked for in~\cite{Kammeyer-et-al:profinite-invariants}*{Section~1.3} where the arithmetic case of Theorem~\ref{thm:sign-profinite} was proven, when \(S_i\) consists of the infinite places only.  The congruence subgroup property, meaning \(\mathbf{G_i}\) has finite \(S_i\)-congruence kernel, implies that the congruence and profinite completions are commensurable, so that Theorem~\ref{thm:sign-profinite} is immediate from Theorem~\ref{thm:main-theorem}.  According to a conjecture of Serre, all higher rank \(S\)-arithmetic groups should have CSP.  The status of this conjecture is advanced~\cite{Prasad-Rapinchuk:developments}: it is known to hold true for all isotropic forms and is currently open only for certain anisotropic forms of type \(A_n\), \(D_4\), and~\(E_6\).

\medskip
Let us give some comments on the methods to prove Theorem~\ref{thm:main-theorem}.  The previous proof of the arithmetic case of Theorem~\ref{thm:sign-profinite} with \(k_i = \Q\) in~\cite{Kammeyer-et-al:profinite-invariants} used that \(\operatorname{sgn} \chi (\Gamma_i)\), if nonzero, can be identified with \(\frac{\dim X_i}{2}\)~mod~\(2\) where \(X_i\) is the symmetric space with isometry group \(\mathbf{G_i}(\R)\).  Equality of the signs was then concluded from a Gauss sum formula, applied to the Killing form, which shows that the signature mod~\(8\) is determined by the \(\Q_p\)-completions of the form only.

The argument, however, does not extend to the \(S\)-arithmetic case in any apparent way because an \(S\)-arithmetic group acts with finite covolume on a product of symmetric spaces and Bruhat--Tits buildings.  This implies that \(\operatorname{sgn} \chi (\Gamma_i)\), if nonzero, is the same as the mod~\(2\) invariant
\[ d(\mathbf{G_i}) = \sum_{v \textup{ real}} \frac{\dim X^v_i}{2} + \sum_{v \in S_i, \,v \nmid \infty} \operatorname{rank}_{{k_i}_v} \mathbf{G_i} \mod 2 \]
where \(X^v_i\) denotes the symmetric space associated with \(\mathbf{G_i}({k_i}_v)\).  So we have to understand the interplay of the local forms of \(\mathbf{G_i}\) at archimedean and non-archimedean places.  To do so, we follow a new approach and rely on \emph{Poitou--Tate duality}, a local-global principle for the Galois cohomology of finite abelian modules, from which we establish a relation of the local \emph{Brauer--Witt invariants} of the groups \(\mathbf{G_1}\) and \(\mathbf{G_2}\) at the places \(v \in S_i\).  The exact form of this relation differs depending on the outer type of \(\mathbf{G_i}\) and we go through the census of quasi-split groups to check that the relation indeed implies \(d(\mathbf{G_1}) = d(\mathbf{G_2})\) except possibly in type \({}^6 D_4\).

The results on Brauer--Witt invariants are also of independent interest as they imply necessary and sufficient conditions under which a family of prescribed local isomorphims types is realized by a global group, see for example Theorem~\ref{thm:ses-a2nm1} below for type \({}^2 A_{2n-1}\).  This makes some of the results in~\cite{Prasad-Rapinchuk:existence} more precise.

As part of the case by case study, we single out the types of \(\mathbf{G_i}\) in which some value of \(\operatorname{sgn} \chi(\Gamma_i)\) can be excluded right away.  We gather this information in the following theorem, as we could not find this being worked out in the literature.

\begin{theorem} \label{thm:signs}
 \(S\)-arithmetic subgroups of simple \(k\)-groups of type 
 \begin{enumerate}[label=(\roman*)]
   \setlength\itemsep{2mm}
    \item  \( {}^1 A_n\ (n \ge 2), {}^1 D_{2n+1}\ (n \ge 2), {}^1 E_6 \) \\[1mm]
have zero Euler characteristic,
\item
  \( {}^2 A_{4n}\ (n \ge 1), C_{4n}\ (n \ge 1), {}^1 D_{4n}\ (n \ge 1), {}^3 D_4, {}^2 E_6, E_8, F_4, G_2  \) \\[1mm]
  have either zero or positive Euler characteristic,
\item \label{item:allsigns} \( A_1, {}^2 A_n\ (4 \nmid n \ge 2), B_n\ (n \ge 2), C_n\ (4 \nmid n \ge 3),\) \\
  \( {}^1 D_{4n+2}\ (n \ge 1) , {}^2 D_n \ (n \ge 4), {}^6 D_4, E_7 \) \\[1mm]
  can have zero, positive, or negative Euler characteristic.
  \end{enumerate}
\end{theorem}

So the core value of Theorem~\ref{thm:main-theorem} lies in case~\ref{item:allsigns} of the theorem.

\medskip
Let us point out that in light of \emph{Margulis arithmeticity}, we obtain a version of Theorem~\ref{thm:sign-profinite} in more geometric terms.  By a \emph{higher rank Lie group} \(G\) we refer to the type of generalized Lie groups to which the arithmeticity theorem applies.  This means \(G = \prod_{\alpha \in A} \mathbf{G}_\alpha (k_\alpha)\) where \(A\) is a finite and non-empty set, where \(k_\alpha\) for \(\alpha \in A\) is a local field of characteristic zero (so \(k_\alpha\) is either \(\R\), \(\C\), or a finite degree extension of \(\Q_p\)), where \(\mathbf{G}_\alpha\) is a connected and simply-connected semisimple \(k_\alpha\)-group without \(k_\alpha\)-anisotropic factors, and where \(\sum_{\alpha \in A} \operatorname{rank}_{k_\alpha} \mathbf{G}_\alpha \ge 2\).  For such \(G\), we can define that a lattice \(\Gamma \le G\) is \emph{irreducible} if no finite index subgroup of \(\Gamma\) is a direct product of two infinite subgroups.  The Margulis arithmeticity theorem says that such \(\Gamma \le G\) determines uniquely a number field \(k\), a connected and simply-connected absolutely almost simple \(k\)-group \(\mathbf{G}\), and a finite set of places \(S\) of \(k\) containing the infinite ones, such that \(\prod_{v \in S} \mathbf{G}(k_v)\) surjects onto \(G\) with compact kernel and such that every \(S\)-arithmetic subgroup of \(\mathbf{G}\) is commensurable with~\(\Gamma\).  We say that \(\Gamma \le G\) \emph{has CSP} if \(\mathbf{G}\) has CSP with respect to \(S\).  With these explanations, Theorem~\ref{thm:sign-profinite} gives the following consequence.

\begin{theorem}
For \(i=1,2\), let \(\Gamma_i \le G_i\) be an irreducible lattice with CSP in a higher rank Lie group with no factor of type \(D_4\).  Suppose \(\widehat{\Gamma_1}\) is commensurable with \(\widehat{\Gamma_2}\).  Then \(\operatorname{sgn} \chi(\Gamma_1) = \operatorname{sgn} \chi(\Gamma_2)\).
\end{theorem}

\medskip
The outline of this article is as follows.  In Section~\ref{section:preliminaries}, we collect some preliminaries.  In particular, we recall the definition of the \(S\)-congruence completion, we explain how Theorem~\ref{thm:main-theorem} implies Theorem~\ref{thm:sign-profinite}, and we discuss why the invariant \(d(\mathbf{G_i})\) gives essentially the sign of the Euler characteristic.  In Section~\ref{section:brauer}, we present the Galois cohomological methods of the proof of Theorem~\ref{thm:main-theorem}, establishing that ``the sum'' of local Brauer--Witt invariants of \(\mathbf{G_1}\) and \(\mathbf{G_2}\) over the places in \(S_i\) must agree ``mod~\(2\)''.  In Section~\ref{section:adelic}, we discuss some consequences of this principle which hold for the groups \(\mathbf{G_1}\) and \(\mathbf{G_2}\) in general.  Finally, Section~\ref{section:cartan-types} completes the proof of Theorem~\ref{thm:main-theorem} by verifying that in all relevant inner and outer Cartan--Killing types our assumptions give \(d(\mathbf{G_1}) = d(\mathbf{G_2})\).  We also point out in which types the Euler characteristic is always zero and in which types we exclusively have \(d(\mathbf{G_i}) = 0\) from which Theorem~\ref{thm:signs} follows.

\medskip
Both authors are grateful for financial support from the Research Training Group ``Algebro-Geometric Methods in Algebra, Arithmetic, and Topology'' (DFG 284078965).  H.\,K.\,acknowledges additional funding from the Priority Program ``Geometry at Infinity'' (DFG 441848266).  We wish to thank our colleagues A.\,Baumann, D.\,Echtler, I.\,Halupczok, B.\,Klopsch, and G. Mantilla-Soler for helpful discussions.

\section{Preliminaries} \label{section:preliminaries}

Let \(k\) be a number field, meaning a finite degree extension of the rational numbers \(\Q\).  Consider a finite set \(S\) of places of \(k\), where a \emph{place} is an equivalence class of absolute values on \(k\), and we agree that \(S\) should always contain all archimedean (real and complex) places.  The set \(S\) defines the \emph{ring of \(S\)-integers} \(\mathcal{O}_{k,S} \subset k\) consisting of all elements \(x \in k\) which have absolute value \(\le 1\) at places outside \(S\).  In the special case that \(S\) consists of the archimedean places only, \(\mathcal{O}_{k,S}\) is the ordinary ring of integers \(\mathcal{O}_k\) of \(k\).

\medskip
Let \(\mathbf{G}\) be a simply-connected, absolutely almost simple, linear algebraic \(k\)-group.  The group of \(k\)-rational points \(\mathbf{G}(k)\) comes with two natural Hausdorff topologies: The \emph{arithmetic topology} has all finite index subgroups as a unit neighborhood base.  The \emph{\(S\)-congruence topology} has all \emph{principal \(S\)-congruence subgroups} as unit neighborhood base.  Principal \(S\)-congruence subgroups are of the form \(\ker (\mathbf{G}(\mathcal{O}_{k,S}) \rightarrow \mathbf{G}(\mathcal{O}_{k,S} / \mathfrak{a}))\) for a nonzero ideal \(\mathfrak{a} \trianglelefteq \mathcal{O}_{k,S}\).  Here, we have fixed an embedding \(\mathbf{G} \subset \mathbf{GL_n}\) of which the \(S\)-congruence topology is independent.  As topological groups have a natural uniform structure, we obtain completions \(\widehat{\mathbf{G}}\) and \(\overline{\mathbf{G}}\) of \(\mathbf{G}(k)\) with respect to the arithmetic and \(S\)-congruence topology, respectively. Since the arithmetic topology is a priori finer than the \(S\)-congruence topology, we have a canonical map \(\widehat{\mathbf{G}} \rightarrow \overline{\mathbf{G}}\) whose kernel \(C(\mathbf{G},S)\) is called the \emph{\(S\)-congruence kernel} of \(\mathbf{G}\) with respect to \(S\).  We say that \(\mathbf{G}\) has the \emph{congruence subgroup property} or simply \emph{CSP} with respect to \(S\) if \(C(\mathbf{G}, S)\) is finite.  The congruence kernel \(C(\mathbf{G},S)\) also lies in the short exact sequence
  \[ 1 \longrightarrow C(\mathbf{G},S) \longrightarrow \widehat{\mathbf{G}(\mathcal{O}_{k,S})} \longrightarrow \overline{\mathbf{G}(\mathcal{O}_{k,S})} \longrightarrow 1 \]
  where the completions of the last two groups are defined as before.  We see that any finite index subgroup of \(\widehat{\mathbf{G}(\mathcal{O}_{k,S})}\) which intersects \(C(\mathbf{G},S)\) trivially is mapped isomorphically onto a finite index subgroup of \(\overline{\mathbf{G}(\mathcal{O}_{k,S})}\), so CSP implies that the profinite and \(S\)-congruence completions of \(\mathbf{G}(\mathcal{O}_{k,S})\) are commensurable.  Similarly, CSP gives the commensurability of profinite and \(S\)-congruence completions for every \emph{\(S\)-arithmetic group}, meaning any group \(\Gamma \subset \mathbf{G}(k)\) which is commensurable with \(\mathbf{G}(\mathcal{O}_{k,S})\) (this condition again being independent of the chosen embedding \(\mathbf{G} \subset \mathbf{GL_n}\)).  This shows that Theorem~\ref{thm:main-theorem} implies Theorem~\ref{thm:sign-profinite}.

    \medskip
  Any \(S\)-arithmetic subgroup \(\Gamma\) of \(\mathbf{G}\) is a lattice in the locally compact group given by the product
  \[ G = \prod_{v \in S} \mathbf{G}({k}_v). \]
Such a group is known to carry an \emph{Euler--Poincar\'e measure} \(\mu_G\).  This measure satisfies \(\chi(\Gamma) = \mu_G(G/\Gamma)\) for all \(S\)-arithmetic subgroups \(\Gamma \subset \mathbf{G}\) as explained in \cite{Serre:cohomologie-discrets}*{Th\'eor\`eme~10}.  In particular, the sign of \(\mu_G\) in \(\{-1, 0, 1\}\) agrees with the sign of the Euler characteristic of any \(\Gamma\).  The measure \(\mu_G\) is in fact a product measure \(\mu_G = \otimes_{v \in S} \,\mu_{\mathbf{G}(k_v)}\) corresponding to the above product decomposition.  At each archimedean place \(v\) in \(S\), the measure \(\mu_{\mathbf{G}(k_v)}\) is nonzero if and only if the \emph{fundamental rank} \(\delta(\mathbf{G}(k_v)) = \operatorname{rank}_\C (\mathfrak{g} \otimes_\R \C) - \operatorname{rank}_\C (\mathfrak{k} \otimes_\R \C)\) vanishes, where \(\mathfrak{g}\) and \(\mathfrak{k}\) denote the Lie algebras of \(\mathbf{G}(k_v)\) and of a maximal compact subgroup \(K \le \mathbf{G}(k_v)\), respectively.  Consequently, the fundamental rank is always positive if \(v\) is complex so that \(\mu_G\) vanishes unless the number field \(k\) is totally real.  If on the other hand \(v \in S\) is a real place such that \(\mathbf{G}(k_v)\) has vanishing fundamental rank, then the dimension of the \emph{symmetric space} \(\mathbf{G}(k_v)/K\) is an even number \(2r\) and the sign of \(\mu_{\mathbf{G}(k_v)}\) is \((-1)^r\).  If \(v\) is finite, then \(\mu_{\mathbf{G}(k_v)}\) is always nonzero and the sign is equal to \((-1)^r\) where this time \(r\) is the \(k_v\)-rank of \(\mathbf{G}\).  Thus the sign of \(\mu_G\), if nonzero, is given by \((-1)\) to the power of
  \[ d(\mathbf{G}) = \sum_{v \textup{ real}} \frac{\dim X^v}{2} + \sum_{v \in S, \,v \nmid \infty} \operatorname{rank}_{k_v} \mathbf{G} \mod 2. \]

  To prove Theorem~\ref{thm:main-theorem}, we will first show (in Corollary~\ref{cor:euler-zero}) that \(\chi(\Gamma_1) = 0\) if and only if \(\chi(\Gamma_2) = 0\).  In view of the above formula, the proof of Theorem~\ref{thm:main-theorem} will then be completed by showing in Section~\ref{section:cartan-types} that \(d(\mathbf{G_1}) = d(\mathbf{G_2})\) under the assumption that \(k\) is totally real and that for every real place \(v\) of \(k\) we have \(\delta(\mathbf{G_i}(k_v)) = 0\).

  \section{Brauer--Witt invariants and Poitou--Tate duality} \label{section:brauer}

  Still let \(\mathbf{G}\) be a simply-connected, absolutely almost simple, linear algebraic \(k\)-group and let \(\mathbf{G_0}\) be the up to \(k\)-isomorphism unique \(k\)-quasi-split group of which \(\mathbf{G}\) is an inner \(k\)-twist.  We obtain a finite \(k\)-group \(\mathbf{A_0}\) defined as the factor group in the short exact sequence
\[ 1 \rightarrow \Ad \mathbf{G_0} \rightarrow \Aut \mathbf{G_0} \rightarrow \mathbf{A_0} \rightarrow 1 \]
where \(\Ad \mathbf{G_0}\) is the adjoint form of \(\mathbf{G_0}\).  Note that \(\mathbf{A_0}\) is either trivial or the constant group scheme \(\Z/2\) unless \(\mathbf{G}\) has type \(D_4\).  As explained in~\cite{Serre:galois-cohomology}*{Section~I.5.5}, we obtain an induced exact sequence in Galois cohomology and the group \(\mathbf{G}\) corresponds to a unique class \(\xi\) in the image of the map
\[ H^1(k, \Ad \mathbf{G_0}) \longrightarrow H^1(k, \Aut \mathbf{G_0}). \]
The fiber of \(\xi\) is an \(\mathbf{A_0}(k)\)-orbit \(\eta\) in the set \(H^1(k, \Ad \mathbf{G_0})\).  The interpretation of \(\Ad \mathbf{G_0}\) as inner automorphism group yields the short exact sequence
\[ 1 \rightarrow Z(\mathbf{G_0}) \rightarrow \mathbf{G_0} \rightarrow \Ad \mathbf{G_0} \rightarrow 1 \]
where \(Z(\mathbf{G_0})\) denotes the center of \(\mathbf{G_0 }\).  By~\cite{Serre:galois-cohomology}*{Section~I.5.5}, we obtain a coboundary map in Galois cohomology
\[ \delta^1 \colon H^1(k, \Ad \mathbf{G_0}) \longrightarrow H^2(k, Z(\mathbf{G_0})). \]
Observing that \(\mathbf{A_0}(k)\) acts on \(Z(\mathbf{G_0})\), we obtain an induced \(\mathbf{A_0}(k)\)-action on \(H^2(k, Z(\mathbf{G_0}))\) and \(\delta^1\) is \(\mathbf{A_0}(k)\)-equivariant.  Let \(\overline{\delta^1}\) denote the corresponding orbit map.  Then \(\beta = \overline{\delta^1}(\eta) \in H^2(k, Z(\mathbf{G_0})) / \mathbf{A_0}(k)\) is a well defined invariant of the group \(\mathbf{G}\).  We adopt the terminology used by G.\,Harder in~\cite{Harder:bericht}*{Section~3.2} and call \(\beta\) the \emph{Brauer--Witt invariant} of \(\mathbf{G}\).  We have a diagram
\[
\begin{tikzcd}
\bigoplus_v H^1(k_v, \Ad \mathbf{G_0}) \arrow{r} & \bigoplus_v H^2(k_v, Z(\mathbf{G_0})) \\
{H^1(k, \Ad \mathbf{G_0})} \arrow{r}{\delta^1} \arrow{u} & {H^2(k, Z(\mathbf{G_0}))} \arrow{u}{s}
\end{tikzcd}
\]
where the lower map is surjective by~\cite{Platonov-Rapinchuk:algebraic-groups}*{Theorem~6.20, p.\,326}.  Endowing the upper two Galois cohomology sets with the diagonal action of \(\mathbf{A_0}(k)\) via the embeddings \(\mathbf{A_0}(k) \subset \mathbf{A_0}(k_v)\), all maps are equivariant.

Since \(Z(\mathbf{G_0})\) is a finite commutative group scheme, the \(\overline{k}\)-points form a finite commutative Galois module by functoriality, so that we obtain a description of the image of \(s\) from \emph{Poitou--Tate duality}~\cite{Harari:galois-cohomology}*{Theorem~17.13.(c), p.\,265}.  To explain this, we set for short \(Z = Z(\mathbf{G_0})\) and we let \(Z' = \operatorname{Hom}(Z, \mathbf{GL_1})\) be the \emph{Cartier dual} of \(Z\).  Then we have perfect local Tate duality pairings induced by the cup product
\[ H^0(k_v, Z') \otimes H^2(k_v, Z) \longrightarrow \Q / \Z \]
for finite places \(v\) of \(k\) and
\[ H^0(k_v, Z') \otimes H^2(k_v, Z) \longrightarrow \textstyle \frac{1}{2} \Z / \Z \]
for real places \(v\) of \(k\).  Here only for real places \(v\), the notation \(H^0(k_v, Z')\) denotes the Tate cohomology group, meaning the quotient of the invariants \(Z'^{\operatorname{Gal}(k_v)}\) by the norms \(N_{\operatorname{Gal}(k_v)} Z'\).  The local duality pairings sum up to the perfect Poitou--Tate pairing
\begin{equation} \label{eq:global-poitou-tate} \prod_v H^0(k_v, Z') \,\otimes\, \bigoplus_v H^2(k_v, Z) \longrightarrow \Q / \Z. \end{equation}
The image of \(s\) is precisely the annihilator under this pairing of the image of \(t \colon H^0(k, Z') \longrightarrow \prod_v H^0(k_v, Z')\).  Since \(s\) is moreover injective as proven in~\cite{Platonov-Rapinchuk:algebraic-groups}*{Lemma~6.19, p.\,336} or \cite{Prasad-Rapinchuk:existence}*{p.\,658}, we thus have a short exact sequence
\begin{equation} \label{eq:poitou-tate-ses} 0 \longrightarrow H^2(k, Z) \xrightarrow{\ s\ } \bigoplus_v H^2(k_v, Z) \xrightarrow{\ t^*\ } H^0(k, Z')^* \longrightarrow 0. \end{equation}
where \(t^*\) is the Pontryagin dual of \(t\) obtained by applying the functor \((\,\cdot\,)^* = \operatorname{Hom}(\,\cdot\,, \Q/\Z)\) and the local duality isomorphisms \(H^0(k_v, Z')^* \cong H^2(k_v, Z)\).  Recall that the action of an element \(\sigma \in \operatorname{Gal}(k)\) on \(f \in Z'\) is given by \((\sigma f)(z) = \sigma f(\sigma^{-1} z)\) so that
\[ H^0(k, Z') = Z'^{\operatorname{Gal}(k)} = \operatorname{Hom}_{\operatorname{Gal}(k)}(Z, \mathbf{GL_1}),  \]
meaning a \(k\)-defined character is just a \(\operatorname{Gal}(k)\)-equivariant character.

Returning to our Brauer--Witt invariant, the sequence \eqref{eq:poitou-tate-ses} tells us that every element in the \(\mathbf{A_0}(k)\)-orbit \(\beta\) maps under \(s\) to a family in \(\bigoplus_v H^2(k_v, Z)\) which sums to zero under \(t^*\).  We will next work out what this condition says explicitly for different types of the occurring Galois modules \(Z = Z(\mathbf{G_0})\).

If \(Z(\mathbf{G_0}) = \mu_n\), then \(H^2(k, Z) \cong \Br_n(k)\) is the subgroup of the Brauer group of \(k\) consisting of those elements whose order divides \(n\) and the short exact sequence~\eqref{eq:poitou-tate-ses} becomes
\[ 0 \longrightarrow \Br_n(k) \longrightarrow \bigoplus_v \Br_n(k_v) \longrightarrow \Z/n \longrightarrow 0 \]
which is also immediate from the classical Albert--Brauer--Hasse--Noether theorem.  We have similar results when \(Z \cong \mu_2 \times \mu_2\) or \(Z \cong \mathbf{R}_{l/k}(\mu_2)\) where \(\mathbf{R}_{l/k}\) denotes the restriction of scalars functor for a quadratic extension \(l/k\).  However, one more difficult case, which occurs in particular if \(\mathbf{G_0}\) has type \({}^2A_{2n-1}\), stands out.  In that case we have that \(Z\) is \(k\)-isomorphic to the kernel \(\mathbf{R}^{(1)}_{l/k}(\mu_{2n})\) of the norm map
\[ N \colon \mathbf{R}_{l/k} (\mu_{2n}) \longrightarrow \mu_{2n} \]
for a quadratic extension \(l/k\).  It will turn out that only the case when \(l/k\) is a CM-field, meaning \(k\) is totally real and \(l\) is totally imaginary, needs consideration.  There we have the following result.

  \begin{theorem} \label{thm:ses-a2nm1}
    Let \(l/k\) be a CM-field which defines the Galois module \(Z = \mathbf{R}^{(1)}_{l/k}(\mu_{2n})\).  Then we have a split short exact sequence
    \[ 0 \longrightarrow H^2(k, Z)  \longrightarrow \bigoplus_v H^2(k_v, Z) \longrightarrow \Z / 2 \longrightarrow 0 \]
    with \(H^2(k_v, Z) \cong \Z / 2n\) if \(v\) splits in \(l\) and \(H^2(k_v, Z) \cong \Z / 2\) if \(v\) is non-split.  The map to \(\Z/2\) sums up the \(\Z/2\)-reductions of all coordinates.
  \end{theorem}

  Note that it was shown in \cite{Prasad-Rapinchuk:existence}*{Theorem~3} that in general, the map \(H^2(k, Z) \longrightarrow \bigoplus_{v \neq v_0} H^2(k_v, Z)\) is injective if and only if \(v_0\) does not split in a certain extension \(l/k\) defined by the quasi-split type.  The theorem makes this statement quantitative in type \({}^2A_{2n-1}\): if \(v_0\) splits in \(l\), then the map is \(n\) to one.  We start the proof by determining the abelian groups \(H^2(k_v, Z)\) for the various places \(v\) of \(k\).

  \begin{proposition} \label{prop:galois-action}
    If the place \(v\) of \(k\) splits in \(l\), then \(H^2(k_v, Z) \cong \Z / 2n\).  If \(v\) is inert or ramified in \(l\), then \(H^2(k_v, Z) \cong \Z / 2\).
    \end{proposition}
    
    \begin{proof}
      First, we recall that for every \(l\)-group \(\mathbf{H}\), we have
      \[ \mathbf{R}_{l/k}(\mathbf{H}) \cong \prod_{w \mid v} \mathbf{R}_{l_w/k_v}(\mathbf{H})\]
      as \(\operatorname{Gal}(k_v)\)-modules.  This is stated in~\cite{Platonov-Rapinchuk:algebraic-groups}*{p.\,50} without proof, so let us quickly say that this holds because for every extension \(E/k_v\), we have a chain of natural isomorphisms
      \begin{align*} 
        & \textstyle \mathbf{R}_{l/k}(\mathbf{H}) (E) \cong \mathbf{H}(l \otimes_k E) \cong \mathbf{H}((l \otimes_k k_v) \otimes_{k_v} E) \cong \mathbf{H}((\prod_{w \mid v} l_w) \otimes_{k_v} E) \\
       & \textstyle \cong \mathbf{H}(\prod_{w \mid v} (l_w \otimes_{k_v} E)) \cong \prod_{w \mid v} \mathbf{H}(l_w \otimes_{k_v} E) \cong \prod_{w \mid v} \mathbf{R}_{l_w/k_w}(\mathbf{H}) (E).
      \end{align*}
      In the special case \(\mathbf{H} = \mu_{2n}\), we obtain
      \[ H^1(k_v, \mathbf{R}_{l/k}(\mu_{2n})) \cong \prod_{w \mid v} H^1(k_v, \mathbf{R}_{l_w/k_w} (\mu_{2n})) \cong \prod_{w \mid v} H^1(l_w, \mu_{2n}) \]
      by Shapiro's lemma~\cite{Platonov-Rapinchuk:algebraic-groups}*{p.\,20 and below Lemma~2.3, p.\,73}.  Therefore
      \[ H^1(k_v, \mathbf{R}_{l/k}(\mu_{2n})) \cong \prod_{w \mid v} l_w^* / (l_w^*)^{2n}.\]
      Similarly, we have an isomorphism
      \[H^2(k_v, \mathbf{R}_{l/k}(\mu_{2n})) \cong \prod_{w \mid v} \Br_{2n}(l_w). \]
      Thus the short exact sequence
      \[ 1 \rightarrow \mathbf{R}^{(1)}_{l/k} (\mu_{2n}) \rightarrow \mathbf{R}_{l/k} (\mu_{2n}) \rightarrow \mu_{2n} \rightarrow 1 \]
      induces an exact Galois cohomology sequence
\[ \prod_{w \mid v} l_w^* / (l_w^*)^{2n} \xrightarrow{N_v} k_v^* / (k_v^*)^{2n} \rightarrow H^2(k_v, Z) \rightarrow \prod_{w \mid v} \Br_{2n}(l_w) \xrightarrow{N_v} \Br_{2n}(k_v) \]
which can take one of two forms.  On the one hand, if \(v\) splits in \(l\), meaning \(l \subset k_v\), then the norm maps
\[ N \colon \prod_{w \mid v} l_w^* \longrightarrow k_v^* \quad \text{ and } \quad N \colon \prod_{w \mid v} \Br_{2n}(l_w) \longrightarrow \Br_{2n}(k_v) \]
reduce to the maps \(k_v^* \times k_v^* \longrightarrow k_v^*\) and \(\Br_{2n}(k_v) \times \Br_{2n}(k_v) \rightarrow \Br_{2n}(k_v)\) given by multiplication.  Exactness of the sequence thus shows that
\[ H^2(k_v, Z) \cong \{ (x,x^{-1}) \colon x \in \Br_{2n}(k_v) \} \cong \Br_{2n}(k_v), \]
so abstractly \(H^2(k_v, Z) \cong \Z / 2n\) if \(v\) is split.  Note that real places of \(k\) do not split because \(l\) is a CM-field.

On the other hand, if \(v\) is inert or ramified in \(l\), then \(v\) extends to a unique valuation \(w\) on \(l\).  We obtain the corresponding unique quadratic extension \(l_w / k_v\).  If \(w\) is non-archimedean, then the norm map \(N_v \colon \Br_{2n}(l_w) \longrightarrow \Br_{2n}(k_v)\) is an isomorphism because it extends to the corestriction map \(\Br(l_w) \longrightarrow \Br(k_v)\) which is well-known to be an isomorphism for extensions of local fields, see for instance~\cite{Lorenz:algebra2}*{Satz~10, p.\,311}.  Of course, if \(w\) is complex, then \(\Br_{2n}(l_w)\) is trivial, so \(N_v\) has trivial kernel, too.  Since \(l_w / k_v\) is quadratic, we have \((k_v^*)^{2n} \subseteq N(l_w^*)\), hence exactness of the sequence and local Artin reciprocity show that
\[ H^2(k_v, Z) \cong k_v^* / N(l_w^*) \cong \operatorname{Gal}(l_w/k_v) \cong \Z / 2. \qedhere \]
\end{proof}

Next, we want to determine the group of \(k\)-characters
\[ H^0(k, Z') = \operatorname{Hom}_{\operatorname{Gal}(k)}(\mathbf{R}^{(1)}_{l/k}(\mu_{2n}), \mathbf{GL_1}). \]
To this end, we need to understand the action \(\operatorname{Gal}(k) \curvearrowright \mathbf{R}^{(1)}_{l/k}(\mu_{2n})\).

\begin{proposition}
The Galois module \(\mathbf{R}^{(1)}_{l/k}(\mu_{2n})\) is given by the abelian group \(\mu_{2n}(\overline{k})\) on which \(\sigma \in \operatorname{Gal}(k)\) acts functorially and in addition by inversion if \(\sigma \notin \operatorname{Gal}(l)\).
\end{proposition}

\begin{proof}
Pick \(\alpha \in k\) such that \(l = k(\sqrt{\alpha})\) to compute
\begin{align*}
  l \otimes_k \overline{k} & \,\cong\, k[x]/(x^2 - \alpha) \,\otimes_k\, \overline{k} \,\cong\, \overline{k}[x] / (x^2-\alpha) \\
                           & \cong \overline{k}[x]/(x-\sqrt{\alpha}) \,\times\, \overline{k}[x]/(x + \sqrt{\alpha}) \cong \overline{k} \times \overline{k}
\end{align*}
where the last isomorphism is the product of the evaluation maps in \(\sqrt{\alpha}\) and \(-\sqrt{\alpha}\), respectively.  Thus the Galois action on the {\'e}tale algebra \(l \otimes \overline{k}\) has the following description under this isomorphism.  Every \(\sigma \in \operatorname{Gal}(l) \subset \operatorname{Gal}(k)\) acts diagonally on \(\overline{k} \times \overline{k}\) while every \(\sigma \in \operatorname{Gal}(k) \setminus \operatorname{Gal}(l)\) acts diagonally and by swapping the two coordinates.  We thus have
\[ \mathbf{R}_{l/k}(\mu_{2n})(\overline{k}) \cong \mu_{2n} (l \otimes_k \overline{k}) \cong \mu_{2n}(\overline{k}) \times \mu_{2n}(\overline{k}) \]
where again \(\sigma \in \operatorname{Gal}(k)\) acts diagonally and by swapping if \(\sigma\) does not fix \(l\) pointwise.  Under this isomorphism, the inclusion \(\mathbf{R}_{l/k}(\mu_{2n})(k) \subset \mathbf{R}_{l/k}(\mu_{2n})(\overline{k})\) becomes the diagonal inclusion \(\mu_{2n}(l) \subset \mu_{2n}(\overline{k}) \times \mu_{2n}(\overline{k})\) induced by the two embeddings of \(l\) in \(\overline{k}\).  Hence the norm map
\[ N \colon \mathbf{R}_{l/k}(\mu_{2n})(\overline{k}) \longrightarrow \mu_{2n}(\overline{k}) \]
  corresponds to the product map
  \[ \mu_{2n}(\overline{k}) \times \mu_{2n}(\overline{k}) \longrightarrow \mu_{2n}(\overline{k}) \]
  and \(\mathbf{R}^{(1)}_{l/k}(\mu_{2n})(\overline{k})\) is the subgroup of \(\mu_{2n}(\overline{k}) \times \mu_{2n}(\overline{k})\) given by pairs of conjugate (equivalently inverse) roots of unity.
\end{proof}

Still assuming that \(l/k\) is a CM-field, we abbreviate for simplicity \(\mu_{2n}^{(l)} = \mathbf{R}^{(1)}_{l/k}(\mu_{2n})\) in what follows.  Let us compute the \(k\)-character group
  \[ H^0(k, Z') = \operatorname{Hom}_{\operatorname{Gal}(k)}(\mu_{2n}^{(l)}, \mu_{2n}). \]
  Since \(\operatorname{Hom}(\mu^{(l)}_{2n}, \mu_{2n})\) consists of the maps \(f(\zeta) = \zeta^k\) for \(k = 0, \ldots, 2n-1\) and \(\zeta = \exp(\pi \textup{i} / n)\), the \(\operatorname{Gal}(k)\)-equivariance condition \(f(\sigma(\zeta)) = \sigma(f(\zeta))\) for all \(\sigma \in \operatorname{Gal}(k)\) gives \(\zeta^{-k} = \zeta^k\) or equivalently \(k \in \{0,n\}\).  Thus
  \[ H^0(k, Z') = \{ \zeta \mapsto 1, \zeta \mapsto -1  \} \cong \Z / 2 \]
  and accordingly \(H^0(k_v, Z') = \operatorname{Hom}_{\operatorname{Gal}(k_v)}(\mu_{2n}^{(l)}, \mu_{2n})\) is isomorphic to either \(\Z/2n\) or \(\Z/2\) depending on whether \(v\) splits (meaning \(l \subset k_v\) so that \(\mu^{(l)}_{2n} = \mu_{2n}\) over \(k_v\)) or does not split in \(l\), respectively.  If \(v\) is a real place, recall that we adopted the usual convention that \(H^0(k_v, Z')\) denotes in fact the Tate cohomology group \(\widehat{H}^0(k_v, Z')\) defined as the quotient of the invariants \(Z'^{\operatorname{Gal}(k_v)}\) by the norms \(N_{\operatorname{Gal}(k_v)} Z'\).  Since in a CM-field extension all real places of \(k\) become complex in \(l\), we have
  \[ H^0(k_v, Z') \cong \{ \zeta \mapsto \pm 1 \} / \{ 1 \} \cong \Z / 2 \]
  for each real place \(v\).  We are now readily prepared to prove the theorem.

\begin{proof}[Proof of Theorem~\ref{thm:ses-a2nm1}.]
  By Proposition~\ref{prop:galois-action} and the above calculations, the local Tate duality pairings
  \[ H^0(k_v, Z') \otimes H^2(k_v, Z) \longrightarrow \Q / \Z \]
  are perfect parings of the form
  \begin{gather*}
    \Z / 2n \otimes \Z / 2n \longrightarrow \Z / 2n \subset \Q / \Z \ \ \text{if } v \text{ is split,} \\
    \Z / 2 \otimes \Z / 2 \longrightarrow \Z / 2 \subset \Q / \Z \ \ \text{if } v \text{ is non-split.}
  \end{gather*}
  Hence global Poitou--Tate duality takes the form that the image of
  \[ s \colon H^2(k, Z) \longrightarrow \bigoplus_v H^2(k_v, Z) \]
  is the annihilator of the element \(((n, n, \ldots), (1, 1, \ldots))\) in the product
  \[ \prod_v H^0(k_v, Z') \cong \prod_{v \text{ split}} \Z / 2n \ \times \!\!\! \prod_{v \text{ non-split}} \!\!\!\! \Z / 2. \]
with respect to the pairing~\eqref{eq:global-poitou-tate}.  Equivalently, it is the annihilator of the element \((1, 1, \ldots)\) if we compose the pairings \(\Z / 2n \otimes \Z / 2n \longrightarrow \Z / 2n \subset \Q / \Z\) with the canonical projection \(\Z/2n \rightarrow \Z / 2\).  Together with the injectivity of \(s\) we already mentioned, this completes the proof.
\end{proof}
    
\section{Adelically isomorphic groups} \label{section:adelic}

In this section, we collect some properties that two simple algebraic groups must share if they contain \(S\)-arithmetic subgroups with isomorphic congruence completions.  We denote by \(\mathbb{A}_{k,S}\) the \emph{ring of \(S\)-adeles} of \(k\), meaning the subring of the product \(\prod_{v \notin S} k_v\) consisting of those elements which have almost all coordinates in the valuation ring \(\mathcal{O}_v \le k_v\).  We have a diagonal embedding \(k \le \mathbb{A}_{k, S}\).  It is a well-known consequence of strong approximation~\cite{Platonov-Rapinchuk:algebraic-groups}*{Theorem~7.12, p.\,427} that the congruence completion \(\overline{\Gamma}\) of any infinite \(S\)-arithmetic subgroup \(\Gamma \le \mathbf{G}(k)\) coincides with the closure of \(\Gamma\) in \(\mathbf{G}(\mathbb{A}_{k,S})\) along the embedding \(\Gamma \le \mathbf{G}(k) \le \mathbf{G}(\mathbb{A}_{k,S})\), and this closure is an open compact subgroup of \(\mathbf{G}(\mathbb{A}_{k,S})\).

Returning to the setting of Theorem~\ref{thm:main-theorem}, the commensurability of the congruence completions \(\overline{\Gamma_1}\) and \(\overline{\Gamma_2}\) thus effects that a finite index subgroup of \(\Gamma_1\) embeds into \(\mathbf{G_2}(\mathbb{A}_{k_2,S_2})\) such that the closure of the image has nonempty interior and vice versa.  Thus the assumptions of \emph{adelic superrigidity}~\cite{Kammeyer-Kionke:adelic-superrigidity}*{Theorem~3.2} are satisfied and we conclude as in \cite{Kammeyer-Kionke:adelic-superrigidity}*{Theorem~3.4} that there exists an isomorphism \(\phi \colon  \mathbb{A}_{k_1,S_1} \xrightarrow{\cong} \mathbb{A}_{k_2, S_2}\) over which  \(\mathbf{G_1}\) and \(\mathbf{G_2}\) are isomorphic.  In fact, loc.cit.\ states this in case \(S_i\) consists of the infinite places only.  But the extension to general finite sets of places \(S_i\) is explained in~\cite{KKK:volume}*{Appendix~A}.  Note that there we require that \(S_i\) contains no finite places at which \(\mathbf{G_i}\) is anisotropic which we may assume replacing \(\Gamma_i\) with a finite index subgroup if need be.  In fact, by construction the isomorphism \(\phi\) is assembled from local isomorphisms \(\phi_v \colon {k_1}_v \xrightarrow{\cong} {k_2}_{j(v)}\) where \(j \colon S_1^c \xrightarrow{\cong} S_2^c\) is a bijection between the complements of \(S_1\) and \(S_2\), and the isomorphism \(\mathbf{G_1} \cong_\phi \mathbf{G_2}\) splits into isomorphisms \(\mathbf{G_1} \cong_{\phi_v} \mathbf{G_2}\) for \(v \in S_1^c\).  This shows in particular that \(\mathbf{G_1}\) and \(\mathbf{G_2}\) have the same Cartan--Killing type.

Let \(\mathbf{G_{0,i}}\) be the unique quasi-split \(k_i\)-group of which \(\mathbf{G_i}\) is an inner \(k_i\)-form.  As we explained in the previous section, we have a corresponding finite group scheme \(\mathbf{A_{0,i}}\) of outer automorphisms of \(\mathbf{G_{0,i}}\).  The group \(\mathbf{G_i}\) determines a well-defined Brauer--Witt invariant \(\beta_i \in H^2(k_i, Z(\mathbf{G_{0,i}})) / \mathbf{A_{0,i}}(k_i)\).  Since \(\mathbf{G_1} \cong_{\phi_v} \mathbf{G_2}\) for \(v \in S_1^c\), we have isomorphisms \(\mathbf{G_{0,1}} \cong_{\phi_v} \mathbf{G_{0,2}}\) for \(v \in S_1^c\) such that the induced isomorphisms
\begin{equation} \label{eq:local-iso} H^2({k_1}_v, Z(\mathbf{G_{0,1}}))/\mathbf{A_{0,1}}({k_1}_v) \xrightarrow{\ \cong\ } H^2({k_2}_{j(v)}, Z(\mathbf{G_{0,2}}))/\mathbf{A_{0,2}}({k_2}_{j(v)}) \end{equation}
send the localizations \({\beta_1}_v\) to \({\beta_2}_{j(v)}\).

\medskip
For what comes next, we single out the case when \(\mathbf{G_1}\) and \(\mathbf{G_2}\) have type \({}^2 A_{2n-1}\) or \({}^2 D_n\).  In that case, we always have \(H^0(k_i, Z(\mathbf{G_{0,i}})) \cong \Z/2\) and we verify in the next proposition that the local Tate duality homomorphism descends to a map
 \begin{equation} \label{eq:local-tate-orbit} H^2({k_i}_v, Z(\mathbf{G_{0,i}}))/\mathbf{A_{0,i}}({k_i}_v) \longrightarrow H^0(k_i, Z(\mathbf{G_{0,i}})')^* \cong \Z/2 \end{equation}
 on \(\mathbf{A_{0,i}}({k_i}_v)\)-orbits.  So we obtain well-defined ``mod~2 reductions'' \({\widetilde{\beta_i}}_v  \in \Z/2\) of local Brauer--Witt invariants such that \({\widetilde{\beta_1}}_v = {\widetilde{\beta_2}}_{j(v)}\) for \(v \in S_1^c\):
 
\begin{proposition}
  Let \(\mathbf{G_1}\) and \(\mathbf{G_2}\) be of type \({}^2 A_{2n-1}\) or \({}^2 D_n\).  Then for all \(v \in S_i^c\), the isomorphism~\eqref{eq:local-iso} and the maps~\eqref{eq:local-tate-orbit} form a well-defined commutative triangle
  \[ \begin{tikzcd}[row sep=3mm, column sep=2mm]
      H^2({k_1}_v, Z(\mathbf{G_{0,1}}))/\mathbf{A_{0,1}}({k_1}_v) \ar{rr}{\cong} \ar[dr] &  & H^2({k_2}_{j(v)}, Z(\mathbf{G_{0,2}}))/\mathbf{A_{0,2}}({k_2}_{j(v)}) \ar[dl] \\
        & \Z/2. &
      \end{tikzcd} \]
\end{proposition}

\begin{proof}
  If \(\mathbf{G_1}\) and \(\mathbf{G_2}\) are of type \({}^2 A_{2n-1}\) or \({}^2 D_{2n+1}\), we have \(Z(\mathbf{G_i}) \cong \mathbf{R}^{(1)}_{l/k}(\mu_{2n})\) or \(Z(\mathbf{G_i}) \cong \mathbf{R}^{(1)}_{l/k}(\mu_4)\), respectively.  Suppose \(\mathbf{G_i}\) and hence \(\mathbf{G_{0,i}}\) remains an outer form at \(v \in S^c_i\).  Then we saw in Proposition~\ref{prop:galois-action} that \(H^2({k_i}_v, Z(\mathbf{G_{0,i}})) \cong \Z / 2\).  This implies that the \(\mathbf{A_{0,i}}({k_i}_v)\)-action (by group automorphisms!)\ is trivial and the triangle just consists of the unique isomorphisms between two cyclic groups of order two.

  If on the other hand \(\mathbf{G_{0,i}}\) splits at \(v\), then Proposition~\ref{prop:galois-action} says that \(H^2({k_i}_v, Z(\mathbf{G_{0,i}}))) \cong \Z/2n\) is cyclic, so the map
    \begin{equation} \label{eq:local-tate} H^2({k_i}_v, Z(\mathbf{G_{0,i}})) \longrightarrow H^0(k_i, Z(\mathbf{G_{0,i}})')^* \cong \Z/2 \end{equation}
    must be the unique surjective homomorphism.  It is explained in~\cite{Kneser:galois}*{p.\,254} that the \(\mathbf{A_{0,i}}({k_i}_v)\)-action on \(H^2({k_i}_v, Z(\mathbf{G_{0,i}}))\) is given by inversion so the map is well-defined on \(\mathbf{A_{0,i}}({k_i}_v)\)-orbits.  By uniqueness, the triangle commutes.
    
    \smallskip
    Now assume \(\mathbf{G_1}\) and \(\mathbf{G_2}\) are of type \({}^2 D_{2n}\).  Then \(Z(\mathbf{G_i}) \cong \mathbf{R}_{l/k}(\mu_2)\).  If \(\mathbf{G_{0,i}}\) is an outer form at \(v \in S^c_i\), then \(H^2({k_i}_v, Z(\mathbf{G_{0,i}})) \cong \Z/2\) so the same arguments as above apply.  If \(\mathbf{G_{0,i}}\) splits at \(v\), then decoding the definition of local Tate duality via cup products, the map~\eqref{eq:local-tate} can be canonically identified with the addition \(\Z/2 \oplus \Z/2 \rightarrow \Z/2\).  The \(\mathbf{A_{0,1}}({k_i}_v)\)-action swaps the two coordinates, so the map is well-defined on \(\mathbf{A_{0,i}}({k_i}_v)\)-orbits.  The triangle commutes because the horizontal isomorphism is induced by a \(\Z/2\)-equivariant isomorphism \(H^2({k_1}_v, Z(\mathbf{G_{0,1}})) \xrightarrow{\cong} H^2({k_2}_{j(v)}, Z(\mathbf{G_{0,2}}))\).
\end{proof}

In the types of the proposition, we have seen that the map \(t^*\) in the Poitou--Tate sequence~\eqref{eq:poitou-tate-ses} just sums up the mod~2 reductions, so we obtain the following corollary.

\begin{proposition} \label{prop:mod2sum}
  Let \(\mathbf{G_1}\) and \(\mathbf{G_2}\) be of type \({}^2 A_{2n-1}\) or \({}^2 D_n\).  Then
  \[ \sum_{v \in S_1} {\widetilde{\beta_1}}_v = \sum_{v \in S_2} {\widetilde{\beta_2}}_{j(v)} \]
in \(H^0(k_i,Z') \cong \Z/2\).
\end{proposition}

\begin{proof}
By Poitou--Tate duality, the sum of all reduced local Brauer--Witt invariants is zero for \(\mathbf{G_1}\) and \(\mathbf{G_2}\).  By the previous proposition, the same number of non-trivial summands occur in \(S_1^c\) and \(S_2^c\).
\end{proof}

As another consequence of the isomorphism \(\mathbf{G_1} \cong_\phi \mathbf{G_2}\), we record the following observation.

\begin{proposition} \label{prop:outer-forms-in-s}
  We have \(|S_1| = |S_2|\).  Let moreover \(n_i\) be the number of non-archimedean places in \(S_i\) at which \(\mathbf{G_i}\) is an outer form of order two.  If \(\mathbf{G_i}\) is of triality type \({}^6 D_4\), assume that \(k_i = \Q\).  Then \(n_1 = n_2\).
\end{proposition}

\begin{proof}
  Since we have \(\mathbb{A}_{k_1, S_1} \cong \mathbb{A}_{k_2, S_2}\), the number fields \(k_1\) and \(k_2\) have the same unordered tuples of inertia degrees over almost every rational prime \(p\).  It then follows that they have the same unordered tuples of inertia degrees over every prime \(p\) \cite{Klingen:similarities}*{Theorem~III.1.3, p.\,77}, meaning \(k_1\) and \(k_2\) are \emph{arithmetically equivalent}.  This shows in particular that every rational prime splits into the same number of primes in \(k_1\) and \(k_2\).  Additionally, arithmetically equivalent number fields have the same signature~\cite{Klingen:similarities}*{Theorem III.1.4.h), p.\,79}, meaning the same number of real and complex places, respectively.  So the bijection \(j \colon S_1^c \rightarrow S_2^c\) can be extended to a bijection \(S_1 \rightarrow S_2\) which shows \(|S_1| = |S_2|\).

  Let \(\mathbf{G}\) be the unique simply-connected simple \(\Q\)-split group of the same Cartan--Killing type as \(\mathbf{G_1}\) and \(\mathbf{G_2}\).  Let \(\Delta\) be the Dynkin diagram of \(\mathbf{G}\).  The symmetry group \(\operatorname{Sym} \Delta\) of the diagram can be identified with the outer automorphism group of \(\mathbf{G}\).  Then \(\mathbf{G_i}\) corresponds to a unique class \([a_i] \in H^1(k_i, \operatorname{Aut} \mathbf{G})\) and we let \([b_i]\) be the image of \([a_i]\) in \(H^1(k_i, \operatorname{Sym} \Delta)\).  Since \(\operatorname{Sym} \Delta\) is a trivial Galois module, the cohomology class \([b_i]\) is in fact a conjugacy class of homomorphisms \(\operatorname{Gal}(k_i) \rightarrow \operatorname{Sym} \Delta\).  The Galois extension \(l_i / k_i\) determined by the corresponding kernel is minimal with the property that \(\mathbf{G_i}\) becomes an inner form over \(l_i\).  We claim that also \(l_1\) and \(l_2\) are arithmetically equivalent to one another.  Indeed, for each \(v \in S_1^c\), the fact that \(\mathbf{G_1}\) is isomorphic to \(\mathbf{G_2}\) over the isomorphism \(\phi_v\) means that \(\phi_v\) induces a commutative square
  \[
  \begin{tikzcd}
    H^1({k_1}_v, \operatorname{Aut} \mathbf{G}) \arrow{d}{\cong} \arrow{r}& H^1({k_1}_v, \operatorname{Sym} \Delta) \arrow{d}{\cong} \\
    H^1({k_2}_{j(v)}, \operatorname{Aut} \mathbf{G}) \arrow{r} & H^1({k_2}_{j(v)}, \operatorname{Sym} \Delta)
  \end{tikzcd}
  \]
  such that the vertical maps send \([a_1]\) to \([a_2]\) and \([b_1]\) to \([b_2]\), respectively.  It follows that for each \(v \in S_1^c\), we have \({l_1}_w \cong {l_2}_{w'}\) for all places \(w \mid v\) and \(w' \mid j(v)\), so \(l_1\) and \(l_2\) are arithmetically equivalent.  Now if \([l_i \colon k_i]\) equals one or three, we have \(n_1 = n_2 = 0\).  If \([l_i \colon k_i] = 2\), then \(\mathbf{G_i}\) is an outer form of order two at a non-archimedean place \(v \in S_i\) if and only if \(v\) is inert or ramified in \(l_i\).  Let \(m_i\) be the number of non-archimedean places of \(k_i\) in \(S_i\) and let \(r_i\) be the number of places of \(l_i\) that lie over non-archimedean places of \(k_i\) in \(S_i\).  Then \(n_i = 2m_i - r_i\).  We have \(m_1 = m_2\) by what we discussed above. Since \(l_1\) and \(l_2\) are arithmetically equivalent, we also have \(r_1 = r_2\), hence \(n_1 = n_2\).

  Finally, if \([l_i : k_i] = 6\), then \(\mathbf{G_i}\) has type \({}^6 D_4\), so by assumption \(k_i = \Q\) meaning \(l_i\) has degree six over \(\Q\).  It is then a result of Perlis~\cite{Perlis:equation} that all number fields of degree six or less over \(\Q\) are \emph{arithmetically solitary}, meaning arithmetical equivalence implies isomorphism.  So \(l_1 \cong l_2\), whence \(n_1 = n_2\).
\end{proof}

The fact that the fields \(l_1\) and \(l_2\) from the last proof are arithmetically equivalent allows the following conclusion.

\begin{proposition} \label{prop:same-f-rank}
  The Lie groups \(\prod_{v \mid \infty} \mathbf{G_1}({k_1}_v)\) and \(\prod_{v \mid \infty} \mathbf{G_2}({k_2}_v)\) have the same number of inner real, outer real, and complex factors.
\end{proposition}

\begin{proof}
Recall that arithmetically equivalent fields have the same signature.  Denote the common signature of \(k_1\) and \(k_2\) by \((r,s)\) and denote the common signature of \(l_1\) and \(l_2\) by \((R,S)\).  Then the number of inner real factors in the above Lie group is \(R/2\), the number of outer real factors is \(S-s\) and the number of complex factors is \(s\).
\end{proof}

\begin{proposition}
  The Lie groups \(\prod_{v \mid \infty} \mathbf{G_1}({k_1}_v)\) and \(\prod_{v \mid \infty} \mathbf{G_2}({k_2}_v)\) have the same fundamental rank.
\end{proposition}

\begin{proof}
  By~\cite{Kammeyer-et-al:profinite-invariants}*{Proposition~2,9}, all inner real forms of the same Cartan-Killing type have the same fundamental rank and the same holds for the outer real forms.  The complex factors are isomorphic of course.
\end{proof}

\begin{corollary} \label{cor:euler-zero}
  We have \(\chi(\Gamma_1) = 0\) if and only if \(\chi(\Gamma_2) = 0\).
\end{corollary}

\begin{proof}
  As we saw in Section~\ref{section:preliminaries}, we have \(\chi(\Gamma_i) = 0\) if and only if the Lie group \(\prod_{v \mid \infty} \mathbf{G_i}({k_i}_v)\) has positive fundamental rank.
\end{proof}

\section{The sign of the Euler characteristic by Cartan type} \label{section:cartan-types}

In this section, we will go through all inner and outer Cartan--Killing types to see that the relevant sum formula for the local Brauer--Witt invariants implies the equality \(d(\mathbf{G_1}) = d(\mathbf{G_2})\).  It will turn out that in some types, this will be automatic because \(d(\mathbf{G_i})\) has to vanish anyway and so this section gives the proof of Theorem~\ref{thm:signs} along the way.  As we explained at the end of Section~\ref{section:preliminaries}, Corollary~\ref{cor:euler-zero} allows us to assume \(k_i\) is totally real in what follows and that \(\mathbf{G_i}\) has no real form of positive fundamental rank at any real place.  We will freely use the classification theory of simple algebraic groups by Tits indices as outlined in~\cite{Tits:classification}.  Recall that the fields \(l_1\) and \(l_2\) introduced in the proof of Proposition~\ref{prop:outer-forms-in-s} are arithmetically equivalent, so they have same degree over \(k_1\) and \(k_2\), respectively.  Therefore both \(\mathbf{G_1}\) and \(\mathbf{G_2}\) have the same inner or outer Cartan--Killing type, so we can proceed case by case through the types as follows.

\medskip
\noindent \emph{Type \(E_8\), \(F_4\), \(G_2\).}   If \(\mathbf{G_{0,i}}\) has type \(E_8\), \(F_4\), or \(G_2\), then \(Z(\mathbf{G_{0,i}})\) is trivial, and hence so are all local and global Brauer--Witt invariants.  The Dynkin diagrams of these types do not have symmetries so that for non-archimedean \(v\), the set \(H^1(k_v, \Aut \mathbf{G_{0,i}}) \cong H^1(k_v, \Ad \mathbf{G_{0,i}})\) is a singleton by Kneser's theorem~\cite{Kneser:galois}, meaning that there exists only the split form of these types over \({k_i}_v\).  In particular, all forms have even \({k_i}_v\)-rank.  We infer for instance from~\cite{Helgason:differential-geometry}*{Table~V, p.\,518} that the dimension of each symmetric spaces corresponding to a real form of one of these types is a multiple of four.  Therefore, we always have \(d(\mathbf{G_i})=0\).

\medskip
\noindent \emph{Type \({}^1 A_n\) for \(n \ge 2\), \({}^1 D_{2n+1}\) for \(n \ge 2\), \({}^1 E_6\).}  The inner real forms of these types are (up to isogeny) the groups \(\operatorname{SL}_{n+1}(\R)\), \(\operatorname{SU}^*(n+1)\) for \(n\) odd, \(\operatorname{SO}^0(p,q)\) with \(p\), \(q\) both odd and \(p+q \equiv 2 \mod 4\), \(E_{6(6)}\), and \(E_{6(-26)}\).  All these groups have positive fundamental rank.  Note in particular that the anisotropic real forms are not listed as they are outer forms in these Cartan--Killing types.  Hence there is nothing left to show in this case.

\medskip
\noindent \emph{Type \({}^2 E_6\).}  Since the inner real forms of type \({}^1 E_6\) have positive fundamental rank, only the outer real forms \(E_{6(2)}\), \(E_{6(-14)}\) (and the compact form \(E_{6(-78)}\)) of type \({}^2 E_6\) can occur at the real places of \(k\).  The corresponding symmetric spaces have dimension \(40\) and \(32\) (and \(0\)), so half the dimension is always even.  According to the Tits tables~\cite{Tits:classification}, at the non-archimedean places, either the split form (of rank six) the quasi-split form (of rank four) or a certain inner form of rank two can occur.  So we always have \(d(\mathbf{G_i}) = 0\).

\medskip
\noindent \emph{Type \({}^2 A_{2n}\) for \(n \ge 1\).} The inner real forms of this type have positive fundamental rank, so only the outer ones can occur at archimedean places.  These are (up to isogeny) the groups \(\operatorname{SU}(p,q)\) with \(p+q = 2n+1\).  Hence \(p\) and \(q\) cannot both be odd, so the dimension of the symmetric space \(2pq\) is always a multiple of four.  At any non-archimedean place, the form can either be the outer quasi-split form, which has \({k_i}_v\)-rank \(n\), or it is an inner form of \({k_i}_v\)-rank \(\frac{2n+1}{d}-1\) for some \(d \mid (2n+1)\).  So in the latter case, the \({k_i}_v\)-rank is even.  Thus \(d(\mathbf{G_i}) = 0\) if \(n\) is even.  If \(n\) is odd, then \(d(\mathbf{G_i})\) equals (mod~\(2\)) the number of non-archimedean places \(v \in S\) at which \(\mathbf{G_i}\) is the outer quasi-split form.  But this is just the number of non-archimedean places \(v \in S\) where \(\mathbf{G_{0,i}}\) remains outer, so \(d(\mathbf{G_1}) = d(\mathbf{G_2})\) by Proposition~\ref{prop:outer-forms-in-s}.

\medskip
\noindent \emph{Type \({}^2 A_{2n-1}\) for \(n \ge 2\).} Again, the inner real forms have positive fundamental rank, so only the outer forms \(\operatorname{SU}(p,q)\) with \(p+q = 2n\) can occur at real places.  This means in particular that each real place of \(k_i\) becomes complex in \(l_i\).  In other words \(l_i/k_i\) is a CM-field.  The dimension of the symmetric space \(2pq\) is \(0\) mod~\(4\), or \(2\) mod~\(4\), depending on whether \(p\) and \(q\) are both even or both odd.  At any non-archimedean place \(v\), the form can either be one of the two outer forms with Tits indices
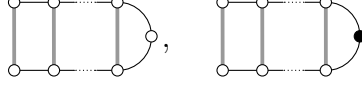
\begin{center}
  \dynkin[scale=1.5,fold]A{oo.ooo.oo}, \quad \dynkin[scale=1.5,fold]A{oo.o*o.oo}
\end{center}
of \({k_i}_v\)-rank \(n\) and \(n-1\), respectively, or it is an inner form of \({k_i}_v\)-rank \(\frac{2n}{d}-1\) for some \(d \mid 2n\).  So each place possibly affects the outcome of \(d(\mathbf{G_i})\).

We first verify that the local Brauer--Witt invariant at a real place \(v\) of \(k_i\) determines half the deminsion of the associated symmetric space mod~\(2\).  By what we just said, the quasi-split group \(\mathbf{G_{0, i}}\) is isomorphic to the unique quasi-split real form \(\mathbf{SU}(n,n)\) at every real place.  Hence at each real place, the exact Galois cohomology sequence takes the form
\begin{align*}
  \mathbf{SU}(n,n)(\R) & \xrightarrow{\pi^0_\R} \mathbf{PSU}(n,n)(\R) \xrightarrow{\delta^0_\R} H^1(\R, Z) \rightarrow H^1(\R, \mathbf{SU}(n,n)) \xrightarrow{\pi^1_\R} \\
  & \rightarrow H^1(\R, \mathbf{PSU}(n,n)) \xrightarrow{\delta^1_\R} H^2(\R, Z)
\end{align*}
with \(Z = Z(\mathbf{SU}(n,n))\) denoting the center.  The Lie group \(\mathbf{SU}(n,n)(\R)\) is connected while \(\mathbf{PSU}(n,n)(\R)\) has two connected components.  Thus \(\delta^0_\R\) is surjective because \(H^1(\R,Z) \cong \{\pm 1\}\), as we infer from \cite{Platonov-Rapinchuk:algebraic-groups}*{(6.30), p.\,332}.  It follows that \(\pi^1_\R\) has trivial kernel but it is not injective.  In fact, the set \(H^1(\R, \mathbf{SU}(n,n))\) classifies isometry classes of nonsingular hermitian forms on the complex vector space \(\C^{2n}\) with discriminant one~\cite{Knus-et-al:involutions}*{p.\,403} (note the sign convention implemented there gives that the split hermitian form \((n,n)\) has discriminant one).  The map \(\pi^1_\R\) sends a hermitian form to the class of the corresponding unitary group, hence precisely the forms with signature \((p,q)\) and \((q,p)\) have the same image.  Since we saw above that \(H^2(\R, Z) \cong \Z / 2\), we conclude from exactness that \(\delta^1_\R\) sends the class of the real form \(\mathbf{SU}(p,q)\) to the discriminant of the underlying hermitian form which is \((-1)^p = (-1)^q = (-1)^{pq}\) because \(p+q=2n\) is even.  Thus at each real place, the Brauer--Witt invariant of a real form agrees mod~\(2\) with the middle dimension of the symmetric space of the real form.

Now let \(v\) be a non-archimedean place.  If \(v\) does not split in \(l\), there exist only two \({k_i}_v\)-forms whose \({k_i}_v\)-rank differs by one.  So we see that the Brauer--Witt invariant also determines the \({k_i}_v\)-rank mod~\(2\) at such places.

Finally, for a split non-archimedean place \(v\), it can be inferred for instance from~\cite{Pierce:associative}*{17.10, Corollary a.(iii), p.\,339} that the Brauer--Witt invariant \(\beta_i \in H^2({k_i}_v, Z) \cong \Br_{2n}({k_i}_v) \cong \Z / 2n\) is \(\frac{2n}{d}\) where \(d\) is the Schur index of the \({k_i}_v\)-division algebra \(D\) in the corresponding \({k_i}_v\)-form \(\mathbf{SL}_{r+1}(D)\) where \((r+1)d = 2n\).  Thus the Brauer--Witt invariant determines the \({k_i}_v\)-rank \(r = \frac{2n}{d}-1\) mod~\(2n\) and hence the Brauer--Witt invariant reduced to \(\Z / 2\) determines the \({k_i}_v\)-rank mod~\(2\).  In conclusion, we obtain \(d(\mathbf{G_1}) = d(\mathbf{G_2})\) from Proposition~\ref{prop:mod2sum}.

\medskip
\noindent \emph{Type \({}^2D_{2n+1}\)}.  In this case, we have \(Z(\mathbf{G_{0,i}}) \cong \mathbf{R}^{(1)}_{l_i/k_i}(\mu_4)\).  The inner real forms of this type have positive fundamental rank.  Hence, we can again assume that \(l_i/k_i\) is a CM-field, so Theorem~\ref{thm:ses-a2nm1} applies with \(n=2\).  The outer real forms are the groups \(\operatorname{SO}^0(p,q)\) with \(p+q = 4n+2\) and both \(p\) and \(q\) even and the form \(\operatorname{SO}^*(4n+2)\).  In the first case, the dimension of the associated symmetric space \(pq\) is thus a multiple of four, whereas in the second case, the dimension is \(4n^2 +2n\), so half the dimension is \(n\) mod~\(2\).  At a non-archimedean place \(v\) of \(k_i\) that is inert or ramified in \(l_i\), the group \(\mathbf{G_i}\) is either the quasi-split form or a certain non-quasi-split outer form:
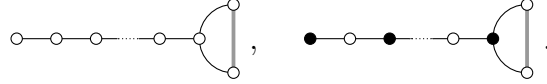
\begin{center}
  \dynkin[scale=1.5, fold]D{ooo.oooo}, \quad \dynkin[scale=1.5, fold]D{*o*.o*oo}.
\end{center}
At a non-archimedean place \(v\) that splits in \(l\), the form has either of the following three Tits indices
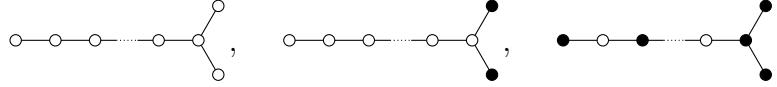
\begin{center}
  \dynkin[scale=1.5]D{ooo.oooo}, \quad \dynkin[scale=1.5]D{ooo.oo**}, \quad \dynkin[scale=1.5]D{*o*.o***}.
\end{center}
If \(n\) is even, we thus see that the \({k_i}_v\)-rank is even for outer forms and odd for inner forms and half the dimension of the symmetric space is even.  Thus \(d(\mathbf{G_i})\) counts (mod 2) the number of non-archimedean places in \(S\) where \(\mathbf{G_{0,i}}\) becomes an inner type, so \(d(\mathbf{G_1}) = d(\mathbf{G_2})\) by Proposition~\ref{prop:outer-forms-in-s}.  If \(n\) is odd, then the two outer forms at a non-split \(v\) are distinguished by the Brauer--Witt invariant.  At a split place \(v\), the group \(\mathbf{A_0}({k_i}_v) \cong \Z / 2\) acts on \(H^2({k_i}_v, Z(\mathbf{G_{0,i}})) \cong \Z/4\) by inversion and the split form has Brauer--Witt invariant \(0\), the second form has Brauer--Witt invariant \(2\), and the third form has Brauer--Witt invariant \(\pm 1\).  This is immediate from the calculation in~\cite{Knus-et-al:involutions}*{Example~31.11, p.\,428}.  The same calculation shows that the real form \(\operatorname{SO}^*(4n+2)\) has non-trivial Brauer--Witt invariant while the real forms \(\operatorname{SO}^0(p,q)\) with \(p+q= 4n+2\) have trivial Brauer--Witt invariant.  Thus the mod~\(2\) reduction of the Brauer--Witt invariant determines the \({k_i}_v\)-rank mod~\(2\) at all non-archimedean places and half the dimension of the symmetric space at all real places.  So \(d(\mathbf{G_1}) = d(\mathbf{G_2})\) follows from Proposition~\ref{prop:mod2sum}.

\medskip
\noindent \emph{Type \({}^1D_{2n}\).} The real forms of this type are the forms \(\operatorname{SO}^0(p,q)\) with \(p+q=4n\) and \(p, q\) even and the form \(\operatorname{SO}^*(4n)\).  In the first case, the dimension of the symmetric space \(pq\) is a multiple of four whereas in the second case the dimension is \(4n^2-2n\) so that half the dimension equals \(n\) mod~\(2\).  At non-archimedean places, the possible forms have Tits indices
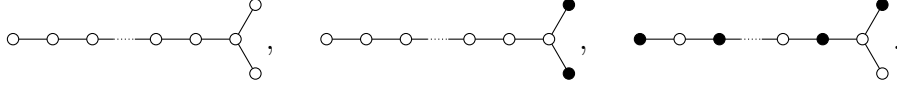
\begin{center}
  \dynkin[scale=1.5]D{ooo.ooooo}, \quad \dynkin[scale=1.5]D{ooo.ooo**}, \quad \dynkin[scale=1.5]D{*o*.o*o*o}.
\end{center}
In the first two cases, the \({k_i}_v\)-rank is even and in the last case, it is~\(n\).  So if \(n\) is even, then \(d(\mathbf{G_1}) = d(\mathbf{G_2}) = 0\).  If \(n\) is odd, we argue as follows.  The center in type \({}^1D_{2n}\) is given by \(Z = Z(\mathbf{G_{0,i}}) \cong \mu_2 \times \mu_2\).  At a non-archimedean place \(v\), the group \(\mathbf{A_0}({k_i}_v) \cong \Z/2\) acts on \(H^2({k_i}_v, Z) \cong \Br_2({k_i}_v) \times \Br_2({k_i}_v) \cong \Z / 2 \times \Z / 2\) by swapping the two factors.  Again, ~\cite{Knus-et-al:involutions}*{Example~31.11, p.\,428} shows that the Brauer--Witt invariant of the first Tits index is \((0,0)\), for the second one it is \((1,1)\) and for the third one it is \(\{(1,0),(0,1)\}\).  Similarly as before, the real form \(\operatorname{SO}^*(4n)\) has Brauer--Witt invariant \(\{(1,0),(0,1)\}\) while the other real forms have Brauer--Witt invariant \((0,0)\) or \((1,1)\).  Poitou--Tate duality now shows that the local Brauer--Witt invariants \((\beta_i)_v\) have to be annihilated by the diagonal image of the multiplication homomorphism \(\mu_2 \times \mu_2 \rightarrow \mu_2\) under \(H^0(k, Z')  \rightarrow \prod_v H^0(k_v, Z')\).  In other words the sum of all coordinates of all local Brauer--Witt invariants has to be zero.  Note that this condition is well-defined.  Since the Brauer--Witt invariant determines the \({k_i}_v\)-rank mod~\(2\) at non-archimedean places and the middle dimension of the symmetric space mod~\(2\) at real places, we conclude \(d(\mathbf{G_1}) = d(\mathbf{G_2})\).

\medskip
\noindent \emph{Type \({}^2D_{2n}\).} We have \(Z(\mathbf{G_{0,i}}) = \mathbf{R}_{l_i/k_i}(\mu_2)\).  The outer real forms of this type are the groups \(\operatorname{SO}^0(p,q)\) with \(p+q=4n\) and \(p,q\) odd.  We can discard them as they have positive fundamental rank.  So \(l_i\) is a totally real extension of \(k_i\).  The inner real types were described above.  At non-archimedean places, the inner types are listed above. There are two outer types at non-archimedean places with Tits indices
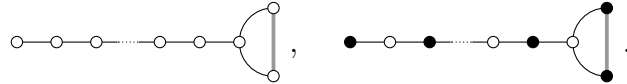
\begin{center}
  \dynkin[scale=1.5, fold]D{ooo.ooooo}, \quad \dynkin[scale=1.5, fold]D{*o*.o*o**}.
\end{center}
If \(n\) is even, we thus see that \(d(\mathbf{G_i})\) counts mod~\(2\) the non-archimedean places where \(\mathbf{G_{0,i}}\) remains outer, so \(d(\mathbf{G_1}) = d(\mathbf{G_2})\) by Proposition~\ref{prop:outer-forms-in-s}.  If \(n\) is odd, we see one more time that the local Brauer--Witt invariants determine the dimension of the symmetric space mod~\(2\) at real places and the \({k_i}_v\)-rank mod~\(2\) at non-archimedean places.  We conclude \(d(\mathbf{G_1}) = d(\mathbf{G_2})\) from Proposition~\ref{prop:mod2sum}.

\medskip
\noindent \emph{Type \({}^3D_4\).}  Clearly, such a form cannot become a \({}^2 D_4\)-form over any field extension.  So only inner types can occur at real places.  The inner real forms of type \(D_4\) are the groups \(\operatorname{SO}^0(4,4)\), \(\operatorname{SO}^0(6,2)\), \(\operatorname{SO}(8)\) with symmetric spaces of dimension \(16\), \(8\), and \(0\), respectively.  Over a non-archimedean place \(v\), the following Tits indices can occur:
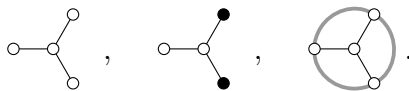
\begin{center}
    \dynkin[scale=1.5]D{oooo}, \quad \dynkin[scale=1.5]D{oo**}, \quad \dynkin[scale=1.5,ply=3]D{oooo}.
\end{center}
The first has \({k_i}_v\)-rank four, the second and third have \({k_i}_v\)-rank two.  Thus always \(d(\mathbf{G_i}) = 0\).

\medskip
\noindent \emph{Type \({}^6D_4\).}  As opposed to the last case, this type can reduce to a type~\({}^2 D_4\)-form locally.  This means that a priori, the real forms \(\operatorname{SO}^0(5,3)\) and \(\operatorname{SO}^0(7,1)\) can occur but these have fundamental rank one, so they can be discarded.  In addition to the three Tits indices depicted above, also the following two indices can occur at a non-archimedean place \(v\):
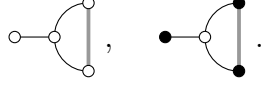
\begin{center}
    \dynkin[scale=1.5, fold]D{oooo}, \quad \dynkin[scale=1.5, fold]D{*o**}.
\end{center}
These have \({k_i}_v\)-rank \(3\) and \(1\), respectively.  Hence \(d(\mathbf{G_i})\) equals mod~\(2\) the number of non-archimedean places \(v \in S\) at which \(\mathbf{G_i}\) is of type \({}^2D_4\).  We conclude \(d(\mathbf{G_1}) = d(\mathbf{G_2})\) from Proposition~\ref{prop:outer-forms-in-s} if \(k_i = \Q\).

\medskip
\noindent \emph{Type \(A_1\), \(B_n\) for \(n \ge 2\), \(C_n\) for \(n \ge 3\), \(E_7\).}  In that case, we have that \(Z(\mathbf{G_{0,i}}) = \mu_2\) and \(\mathbf{A_0}\) is trivial.  Accordingly, \(H^2(k_i, Z(\mathbf{G_{0,i}})) \cong \Br_2(k_i)\) is the subgroup of the Brauer group of elements of order one or two.  As we already mentioned, Poitou--Tate duality thus reduces to the short exact sequence
\begin{equation} \label{eq:abhn} 1 \rightarrow \operatorname{Br}_2({k_i}) \rightarrow \bigoplus_v \operatorname{Br}_2({k_i}_v) \rightarrow \Z / 2 \rightarrow 1 \end{equation}
which also follows from the Albert--Brauer--Hasse--Noether theorem.  So the local Brauer--Witt invariants \(({\beta_1}_v)\) and \(({\beta_2}_v)\) in \(\prod_{v \in S} H^2({k_i}_v, Z(\mathbf{G_{0,i}}))\) have the same coordinate sum in \(\Z / 2\).  We distinguish the cases even further.

\medskip
\noindent \emph{Type \(A_1\).}  If \(\mathbf{G_{0,i}}\) has type \(A_1\), then for all places \(v\), archimedean or not, each coordinate \({\beta_i}_v \in H^2({k_i}_v, Z(\mathbf{G_{0,i}})) \cong \Z/2\) represents either the unique split or the unique ramified quaternion algebra over \({k_i}_v\) according to whether \({\beta_i}_v\) is trivial or non-trivial.  Correspondingly, \(\mathbf{G_i}\) is the unique split or the unique anisotropic form over \({k_i}_v\).  In particular, the \({k_i}_v\)-rank of \(\mathbf{G_i}\) differs by one in these two cases.  If \(v\) is real, we have moreover \(\operatorname{rank}_\R \mathbf{G_i} = \operatorname{rank}_\R \mathfrak{g}_i = \frac{\dim \mathfrak{p}_i}{2}\) where \(\mathfrak{g}_i \cong \mathfrak{k}_i \oplus \mathfrak{p}_i\) is a Cartan decomposition of the Lie algebra \(\mathfrak{g}_i\) of \(\mathbf{G_i}\).  Thus \(d(\mathbf{G_1}) = d(\mathbf{G_2})\) because the local Brauer--Witt invariants for \(v \in S_i\) have the same sum mod~\(2\) for \(\mathbf{G_1}\) and \(\mathbf{G_2}\).

\medskip
\noindent \emph{Type \(B_n\).}  If \(\mathbf{G_{0,i}}\) has type \(B_n\) and \(v\) is non-archimedean, then similarly, \({\beta_i}_v \in H^2({k_i}_v, Z(\mathbf{G_{0,i}})) \cong \Z/2\) corresponds to the unique split or the unique non-split type \(B_n\) form over \({k_i}_v\) which has Tits index
\begin{center}
    \dynkin[scale=2]B{oo.oo*}.
  \end{center}
  So again, the \({k_i}_v\)-rank differs by one in these two cases.  Let us now specify that \(\mathbf{G_{0,i}} = \operatorname{Spin}(q)\) where \(q\) denotes the quadratic form \(q = \langle 1 \rangle ^{n+1} \oplus \langle -1 \rangle^n\).  With this model for \(\mathbf{G_{0,i}}\), we can make the boundary map \(\delta^1 \colon H^1(k_i, \Ad \mathbf{G_{0,i}}) \rightarrow H^2(k_i, Z(\mathbf{G_{0,i}}))\) explicit as \(\delta^1(\alpha) = w_2(q_\alpha) - w_2(q)\) where \(w_2\) denotes the \(\Br_2(k_i)\)-valued Hasse--Witt invariant~\cite{Serre:galois-cohomology}*{Example~3.2.b), p.\,141}.  Here \(q_\alpha\) is the quadratic form obtained from \(q\) by twisting with \(\alpha\), noting that \(\Ad \mathbf{G_{0,i}} = \mathbf{SO}(q)\) and hence \(H^1(k_i, \Ad \mathbf{G_{0,i}})\) classifies quadratic forms over \(k_i\) of rank \(2n+1\) with discriminant \((-1)^n\).  So if \(v\) is real and \(\alpha \in H^1({k_i}_v, \Ad \mathbf{G_{0,i}})\) is such that \(q_\alpha \cong \langle 1 \rangle^{2n+1-p} \oplus \langle -1 \rangle^p\) with \(p \equiv n \mod 2\), then \(\delta^1_v(\alpha)\) is trivial if and only if \(\frac{p(p-1) - n(n-1)}{2}\) is even and we have \(p(p-1) - n(n-1) \equiv p -n \mod 4\).  The symmetric space associated with the Lie group \(\mathbf{SO}(q_\alpha)(\R)\) has dimension \((2n+1-p)p\) and
  \[ (2n+1-p)p \equiv p \mod 4 \]
    if \(n\) and \(p\) are even whereas
    \[ (2n+1-p)p \equiv p + 1 \mod 4 \]
    if \(n\) and \(p\) are odd.  So in any case, the integers \(\frac{p(p-1) - n(n-1)}{2}\) and \(\frac{(2n+1-p)p}{2}\) have constant difference mod~\(2\).  We conclude \(d(\mathbf{G_1}) = d(\mathbf{G_2})\).

\medskip
\noindent \emph{Type \(C_n\).}  If \(\mathbf{G_{0,i}}\) has type \(C_n\) and \(v\) is non-archimedean, then \(\beta^i_v \in H^2(k_v, Z(\mathbf{G_{0,i}})) \cong \Z/2\) informs us on whether the corresponding \({k_i}_v\)-form is split or the unique non-split form with Tits index
\begin{center}
    \dynkin[scale=2]C{*o*o.*o} \quad or \quad \dynkin[scale=2]C{*o*o.o*}
\end{center}
depending on whether \(n\) is even or odd.  The non-split real forms of type \(C_n\) are the groups \(\operatorname{Sp}(p,q)\) and the corresponding symmetric space has dimension \(4pq\) so that half the dimension is always even.  The split real form \(\operatorname{Sp}_n(\R)\) has a symmetric space of dimension \(n(n+1)\).   We now treat the cases \(n \equiv 0, 1, 2, 3 \mod 4\) separately. If \(n \equiv 0\) mod~\(4\), we see that always \(d(\mathbf{G_i}) = 0\).  If \(n \equiv 1\) mod~\(4\), then \(\frac{n(n+1)}{2}\) is odd.  For non-archimedean \(v\), the \({k_i}_v\)-rank of the split form is odd while the \({k_i}_v\)-rank of the non-split form is even.  In type \(C_n\), the local Brauer--Witt invariant is trivial if and only if the corresponding group splits, both for archimedean and non-archimedean places, as can be inferred for instance from~\cite{Adams-Taibi:galois-cohomology}*{Section~10} and \cite{Platonov-Rapinchuk:algebraic-groups}*{Corollary to Theorem~6.20, p.\,326}.  Since the number of places in \(S\) where \(\mathbf{G_i}\) is split (or non-split) is thus equal mod~\(2\) for \(i = 1, 2\), we have \(d(\mathbf{G_1}) = d(\mathbf{G_2})\).  If \(n \equiv 2\) mod~\(4\), then again \(\frac{n(n+1)}{2}\) is odd.  For non-archimedean \(v\), the \({k_i}_v\)-rank of the split form is even and the \({k_i}_v\)-rank of the non-split form is odd.  Thus mod~\(2\), the invariant \(d(\mathbf{G_i})\) is equal to the sum of \([k_i:\Q]\) and the number of places in \(S_i\) where \(\mathbf{G_i}\) is non-split.  Again we conclude \(d(\mathbf{G_1}) = d(\mathbf{G_2})\).  Finally, if \(n \equiv 3\) mod~\(4\), then \(\frac{n(n+1)}{2}\) is even and for non-archimedean \(v\), both the split and non-split form have odd \({k_i}_v\)-rank.  Thus \(d(\mathbf{G_i})\) is always equal to the number of non-archimedean places in \(S_i\), so \(d(\mathbf{G_1}) = d(\mathbf{G_2})\) by Proposition~\ref{prop:outer-forms-in-s} because \(k_1\) and \(k_2\) have the same signature, hence the same number of archimedean places.  This completes the discussion of type \(C_n\).

\medskip
\noindent \emph{Type \(E_7\).}  If \(\mathbf{G_{0,i}}\) has type \(E_7\) and \(v\) is non-archimedean, then \({\beta_i}_v \in H^2({k_i}_v, Z(\mathbf{G_{0,i}}))\) is nontrivial if and only if the corresponding \({k_i}_v\)-form has Tits index
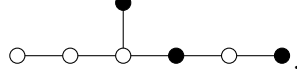
\begin{center}
    \dynkin[scale=2] E{o*oo*o*}.
  \end{center}
  Hence the \({k_i}_v\)-form has odd rank if the Brauer--Witt invariant is trivial and has even rank if the Brauer--Witt invariant is non-trivial.  If on the other hand \(v\) is real, then it is for instance explained in~\cite{Echtler-Kammeyer:bounded}*{Proposition~4} that \(\beta^i_v \in H^2(k_v, Z(\mathbf{G_{0,i}}))\) is trivial if the \({k_i}_v\)-form is split or has Tits index
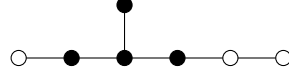
\begin{center}
  \dynkin[scale=2] E{o****oo}
\end{center}
and is nontrivial if the \({k_i}_v\)-form is anisotropic or has the same index as the non-split non-archimedean form above.  In the former case, the symmetric space has dimension \(70\) or \(54\), so half the dimension is odd.  In the latter case, the symmetric space has dimension \(0\) or \(64\), so half the dimension is even.  Hence \(d(\mathbf{G_1}) = d(\mathbf{G_2})\) follows again from the equality mod~\(2\) of the sum of local Brauer invariants for \(v \in S\).

\section{Full triality type} \label{section:full-triality}

In this section, we discuss how the equality \(n_1 = n_2\) in Proposition~\ref{prop:outer-forms-in-s} might fail if \(\mathbf{G_i}\) has type \({}^6 D_4\) over a general number field \(k_i\).  As a consequence, the \(S_i\)-arithmetic subgroups of \(\mathbf{G_i}\) would be profinitely commensurable but have opposite sign of the Euler characteristic.

\medskip
Let \(k\) be a totally real number field of degree \(d \ge 6\) and let \(p\) be a rational prime which is unramified in \(k\) and such that at least one place \(v_0\) of \(k\) over \(p\) has inertia degree three and at least three places \(v_1\), \(v_2\), \(v_3\) of \(k\) over \(p\) have inertia degree one.  For \(i=1,2\), let \(l_i/k\) be a totally real Galois extension with Galois group isomorphic to \(S_3\).  Assume that \(p\) remains unramified in \(l_1\) and \(l_2\).  Suppose moreover that \(v_0\) decomposes into three primes in \(l_1\) and each of \(v_1\), \(v_2\), \(v_3\) decomposes into two primes in \(l_1\).  In contrast, suppose that \(v_0\) splits into six primes in \(l_2\) while \(v_1\), \(v_2\), and \(v_3\) are inert in \(l_2\).  Finally assume that all other \(v \mid p\) decompose equally in \(l_1\) and \(l_2\).  Note that then \(p\) has the same unordered tuple of inertia degrees in \(l_1\) and \(l_2\) which leaves the possibility that \(l_1\) and \(l_2\) are arithmetically equivalent.  The field extension \(l_i/k\) defines a unique conjugacy class of surjective homomorphisms \(\operatorname{Gal}(k) \rightarrow \operatorname{Sym} \Delta\) where \(\Delta\) is the Dynkin diagram of a \(k\)-split group \(\mathbf{G}\) of type \(D_4\).  In other words, we obtain unique elements \(\beta_i \in H^1(k, \operatorname{Sym} \Delta)\) that determine quasi-split \(k\)-groups \(\mathbf{G_1}\) and \(\mathbf{G_2}\) of type \({}^6 D_4\).  Let \(S\) be the set of all places of \(k\) dividing either \(\infty\) or \(p\) and let \(S_i\) be the set of all places of \(l_i\) lying over places in \(S\).  Assuming \(\mathbb{A}_{l_1, S_1} \cong \mathbb{A}_{l_2, S_2}\) over some automorphism of \(\mathbb{A}_{k, S}\), the \(k\)-groups \(\mathbf{G_1}\) and \(\mathbf{G_2}\) are \(S\)-adelically isomorphic but \(n_1 = n_2 + 1\).  

\begin{proposition}
  Suppose there exist number fields \(l_1, l_2 \,/\, k\) and sets of places \(S, S_1, S_2\) as above.  Let \(\Gamma_i\) be an \(S\)-arithmetic subgroup of the \(k\)-quasi-split group \(\mathbf{G_i}\) of type \({}^6 D_4\) defined by \(l_i\).  Then \(\Gamma_1\) is profinitely commensurable with \(\Gamma_2\) but \(\operatorname{sgn} \chi(\Gamma_1) \neq \operatorname{sgn} \chi(\Gamma_2)\).
\end{proposition}

\begin{proof}
  The condition that \(l_1, l_2 \,/\, k\) are totally real gives that the corresponding quasi-split groups \(\mathbf{G_1}\) and \(\mathbf{G_2}\) become the real Lie group \(\operatorname{SO}^0(4,4)\) with fundamental rank zero and symmetric space of dimension~16 at all infinite places of \(k\).  By the description in the previous section, \(d(\mathbf{G_i})\) equals mod~\(2\) the number of non-archimedean places \(v \in S\) such that \(\mathbf{G_i}\) has type \({}^2 D_4\) at \(v\) or, equivalently, such that \(v\) splits into precisely three places in \(l_i\) (each defining a quadratic extension of~\(k_v\)).  Thus indeed \(d(\mathbf{G_1}) \neq d(\mathbf{G_2})\) hence \(\operatorname{sgn} \chi(\Gamma_1) \neq \operatorname{sgn} \chi(\Gamma_2)\).  The isomorphism \(\Phi \colon \mathbb{A}_{l_1, S_1} \xrightarrow{\cong} \mathbb{A}_{l_2, S_2}\) over the automorphism \(\phi \colon \mathbb{A}_{k, S} \xrightarrow{\cong} \mathbb{A}_{k, S}\) induces an isomorphism from \(\mathbf{G_1}\) to \(\mathbf{G_2}\) over \(\phi\) which shows that \(\Gamma_1\) and \(\Gamma_2\) have commensurable congruence completions.  Finally, \(\mathbf{G_i}\) has \(k\)-rank two and thus is known to satisfy the congruence subgroup property.  So \(\Gamma_1\) and \(\Gamma_2\) have commensurable profinite completions.
\end{proof}

  In order to try and find such number fields \(l_1/k\) and \(l_2/k\), one can translate the situation into group theoretic terms via Galois correspondence.  Let \(G \le S_{6d}\) be a transitive permutation group with \(d \ge 6\) and let \(H\) be the stabilizer subgroup of any of the \(6d\) permuted points.  Suppose that some conjugacy class of \(G\) has cycle type containing three cycles of length six and six cycles of length three.  Suppose moreover that \(H\) has a non-trivial \emph{almost conjugate subgroup} \(H'\) in \(G\), meaning a non-conjugate subgroup that intersects each conjugacy class of \(G\) in the same number of elements as \(H\) does.  Assume moreover that there exists a subgroup \(G \ge U \trianglerighteq H, H'\) such that \(U/H\) and \(U/H'\) are isomorphic to \(S_3\).  

  If \(G\) is realizable as Galois group of a totally real Galois extension of \(\Q\), then such a realization defines a totally real number field \(k\) of degree \(d\) over \(\Q\) as the fixed field of \(U\) whereas \(H\) and \(H'\) correspond to the totally real field extensions \(l_1/k\) and \(l_2/k\).  The almost conjugacy of \(H\) and \(H'\) translates into \(l_1\) and \(l_2\) being arithmetically equivalent~\cite{Klingen:similarities}*{Theorem~III.1.3, p.\,77}.  The condition on the cycle type of some conjugacy class in \(G\) ensures by Chebotarev's density theorem that infinitely many primes \(p\) have the decomposition behavior in \(l_1\) and \(l_2\) that would result from the particular decomposition prescribed as above.  One could then check whether the desired decomposition actually occurs over some \(p\) and whether \(\mathbb{A}_{l_1, S_1} \cong \mathbb{A}_{l_2, S_2}\) over an automorphism of \(\mathbb{A}_{k,S}\).

  \medskip
  With \textsc{Magma}, we have checked that for \(d=6\) there exists precisely one and for \(d=7\), there exist precisely two transitive permutation groups \(G \le S_{6d}\) with subgroups \(G \ge U \trianglerighteq H\) such that \(U/H \cong S_3\) and such that \(H\) has a non-trivial almost conjugate subgroup.  Unfortunately, none of these has a conjugacy class with the required cycle type.  The standard \textsc{Magma} distribution only comes with a database of transitive permutation groups up to degree \(47\), so that we could not check the next case \(d=8\).  We remark, however, that recently all 195\,826\,352 conjugacy classes of transitive permutation groups of degree 48 have been enumerated in \textsc{Magma}~\cite{Holt-Royle-Tracey:degree48}.  The database requires more than 30 GB of disk space.  We have not checked it for the groups we are looking for as this would conceivably take an enormous amount of computation time.

  As the number of permutation groups with non-trivial almost conjugate subgroups grows with \(d\), we see no reason why the alleged fields \(l_1, l_2 \,/\, k\) would not exist.  Moreover, the ad hoc example of the decomposition behavior of \(p\) in \(k\) and \(l_i\) is only one of many possibilities to potentially obtain \(n_1 \neq n_2\).  Using ramified primes \(p\), there are even more ways in which the phenomenon \(n_1 \neq n_2\) might arise despite \(\mathbb{A}_{l_1, S_1} \cong \mathbb{A}_{l_2, S_2}\).  So at the time of writing, we have no indication that the non-triality condition in Theorem~\ref{thm:main-theorem} could be removed.


\begin{bibdiv}[References]
  \begin{biblist}

    \bib{Adams-Taibi:galois-cohomology}{article}{
   author={Adams, Jeffrey},
   author={Ta\"{\i}bi, Olivier},
   title={Galois and Cartan cohomology of real groups},
   journal={Duke Math. J.},
   volume={167},
   date={2018},
   number={6},
   pages={1057--1097},
   issn={0012-7094},
   review={\MR{3786301}},
 }
 
\bib{Aka:property-t}{article}{
   author={Aka, Menny},
   title={Profinite completions and Kazhdan's property (T)},
   journal={Groups Geom. Dyn.},
   volume={6},
   date={2012},
   number={2},
   pages={221--229},
   issn={1661-7207},
   review={\MR{2914858}},
 }
 
 \bib{Bridson-Conder-Reid:fuchsian}{article}{
   author={Bridson, M. R.},
   author={Conder, M. D. E.},
   author={Reid, A. W.},
   title={Determining Fuchsian groups by their finite quotients},
   journal={Israel J. Math.},
   volume={214},
   date={2016},
   number={1},
   pages={1--41},
   issn={0021-2172},
   review={\MR{3540604}},
 }

     \bib{Cheetham-West-et-al:property-fa}{article}{
       author={Cheetham-West, Tamunonye},
       author={Lubotzky, Alexander},
       author={Reid, Alan W.},
    author={Spitler, R.},
    title={Property FA is not a profinite property},
   date={2022},
   review={\arXiv{2212.08207}},
 }
 

\bib{Echtler-Kammeyer:bounded}{article}{
    author={Echtler, D.},
    author={Kammeyer, H.},
    title={Bounded cohomology is not a profinite invariant},
    journal={Canad. Math. Bull.},
    note={FirstView},
    pages={1--12},
    date={2023},
   review={\\ \url{https://doi.org/10.4153/s0008439523000826}},
 }

 \bib{Harari:galois-cohomology}{book}{
   author={Harari, David},
   title={Galois cohomology and class field theory},
   series={Universitext},
   note={Translated from the 2017 French original by Andrei Yafaev},
   publisher={Springer, Cham},
   date={2020},
   pages={xiv+338},
   isbn={978-3-030-43901-9},
   isbn={978-3-030-43900-2},
   review={\MR{4174395}},
 }
 
    \bib{Harder:bericht}{article}{
   author={Harder, G\"{u}nter},
   title={Bericht \"{u}ber neuere Resultate der Galoiskohomologie
   halbeinfacher Gruppen},
   language={German},
   journal={Jber. Deutsch. Math.-Verein.},
   volume={70},
   date={1967/68},
   pages={182--216},
   issn={0012-0456},
   review={\MR{0242838}},
}

\bib{Helgason:differential-geometry}{book}{
   author={Helgason, Sigurdur},
   title={Differential geometry, Lie groups, and symmetric spaces},
   series={Graduate Studies in Mathematics},
   volume={34},
   note={Corrected reprint of the 1978 original},
   publisher={American Mathematical Society, Providence, RI},
   date={2001},
   pages={xxvi+641},
   isbn={0-8218-2848-7},
   review={\MR{1834454}},
 }

 \bib{Holt-Royle-Tracey:degree48}{article}{
   author={Holt, Derek},
   author={Royle, Gordon},
   author={Tracey, Gareth},
   title={The transitive groups of degree 48 and some applications},
   journal={J. Algebra},
   volume={607},
   date={2022},
   pages={372--386},
   issn={0021-8693},
   review={\MR{4441332}},
 }
 
 \bib{Jaikin-Lubotzky:grothendieck-pairs}{article}{
       author={Jaikin-Zapirain, Andrei},
       author={Lubotzky, Alexander},
    title={Some remarks on Grothendieck pairs},
   date={2024},
   review={\arXiv{2401.02229}},
 }
 
 \bib{Kammeyer:l2-invariants}{book}{
   author={Kammeyer, Holger},
   title={Introduction to $\ell^2$-invariants},
   series={Lecture Notes in Mathematics},
   volume={2247},
   publisher={Springer, Cham},
   date={2019},
   pages={viii+181},
   isbn={978-3-030-28296-7},
   isbn={978-3-030-28297-4},
   review={\MR{3971279}},
}

\bib{Kammeyer-Kionke:adelic-superrigidity}{article}{
   author={Kammeyer, Holger},
   author={Kionke, Steffen},
   title={Adelic superrigidity and profinitely solitary lattices},
   journal={Pacific J. Math.},
   volume={313},
   date={2021},
   number={1},
   pages={137--158},
   issn={0030-8730},
   review={\MR{4313430}},
 }
 
\bib{KKK:volume}{article}{
   author={Kammeyer, Holger},
   author={Kionke, Steffen},
   author={K\"ohl, Ralf},
   title={Profiniteness of higher rank volume},
   date={2024},
   journal={e-print},
   note={\arXiv{2412.13056}},
 }
 
\bib{Kammeyer-et-al:profinite-invariants}{article}{
   author={Kammeyer, Holger},
   author={Kionke, Steffen},
   author={Raimbault, Jean},
   author={Sauer, Roman},
   title={Profinite invariants of arithmetic groups},
   journal={Forum Math. Sigma},
   volume={8},
   date={2020},
   pages={Paper No. e54, 22},
   review={\MR{4176758}},
 }
 
\bib{Kammeyer-Sauer:spinor}{article}{
   author={Kammeyer, Holger},
   author={Sauer, Roman},
   title={$S$-arithmetic spinor groups with the same finite quotients and
   distinct $\ell^2$-cohomology},
   journal={Groups Geom. Dyn.},
   volume={14},
   date={2020},
   number={3},
   pages={857--869},
   issn={1661-7207},
   review={\MR{4167024}},
 }
 
 \bib{Klingen:similarities}{book}{
   author={Klingen, Norbert},
   title={Arithmetical similarities},
   series={Oxford Mathematical Monographs},
   note={Prime decomposition and finite group theory;
   Oxford Science Publications},
   publisher={The Clarendon Press, Oxford University Press, New York},
   date={1998},
   pages={x+275},
   isbn={0-19-853598-8},
   review={\MR{1638821}},
 }
 
 \bib{Kneser:galois}{article}{
   author={Kneser, Martin},
   title={Galois-Kohomologie halbeinfacher algebraischer Gruppen \"{u}ber
   ${\germ p}$-adischen K\"{o}rpern. II},
   language={German},
   journal={Math. Z.},
   volume={89},
   date={1965},
   pages={250--272},
   issn={0025-5874},
   review={\MR{0188219}},
 }

 \bib{Knus-et-al:involutions}{book}{
   author={Knus, Max-Albert},
   author={Merkurjev, Alexander},
   author={Rost, Markus},
   author={Tignol, Jean-Pierre},
   title={The book of involutions},
   series={American Mathematical Society Colloquium Publications},
   volume={44},
   note={With a preface in French by J. Tits},
   publisher={American Mathematical Society, Providence, RI},
   date={1998},
   pages={xxii+593},
   isbn={0-8218-0904-0},
   review={\MR{1632779}},
 }
 
 \bib{Lorenz:algebra2}{book}{
   author={Lorenz, Falko},
   title={Einf\"{u}hrung in die Algebra. Teil II},
   language={German},
   publisher={Bibliographisches Institut, Mannheim},
   date={1990},
   pages={x+386},
   isbn={3-411-14801-2},
   review={\MR{1115667}},
 }
 
\bib{Lueck:l2-invariants}{book}{
   author={L\"{u}ck, Wolfgang},
   title={$L^2$-invariants: theory and applications to geometry and
   $K$-theory},
   series={Ergebnisse der Mathematik und ihrer Grenzgebiete. 3. Folge. A
   Series of Modern Surveys in Mathematics},
   volume={44},
   publisher={Springer-Verlag, Berlin},
   date={2002},
   pages={xvi+595},
   isbn={3-540-43566-2},
   review={\MR{1926649}},
 }

 \bib{Perlis:equation}{article}{
   author={Perlis, Robert},
   title={On the equation $\zeta \sb{K}(s)=\zeta \sb{K'}(s)$},
   journal={J. Number Theory},
   volume={9},
   date={1977},
   number={3},
   pages={342--360},
   issn={0022-314X},
   review={\MR{0447188}},
 }

 \bib{Pierce:associative}{book}{
   author={Pierce, Richard S.},
   title={Associative algebras},
   series={Studies in the History of Modern Science},
   volume={9},
   note={Graduate Texts in Mathematics, 88},
   publisher={Springer-Verlag, New York-Berlin},
   date={1982},
   pages={xii+436},
   isbn={0-387-90693-2},
   review={\MR{0674652}},
}
 
 \bib{Platonov-Rapinchuk:algebraic-groups}{book}{
   author={Platonov, Vladimir},
   author={Rapinchuk, Andrei},
   title={Algebraic groups and number theory},
   series={Pure and Applied Mathematics},
   volume={139},
   note={Translated from the 1991 Russian original by Rachel Rowen},
   publisher={Academic Press, Inc., Boston, MA},
   date={1994},
   pages={xii+614},
   isbn={0-12-558180-7},
   review={\MR{1278263}},
 }

  \bib{Prasad-Rapinchuk:developments}{article}{
       author={Prasad, Gopal},
       author={Rapinchuk, Andrei},
    title={Developments on the congruence subgroup problem after the work of Bass, Milnor and Serre},
   date={2008},
   review={\arXiv{0809.1622}},
 }
 
 \bib{Prasad-Rapinchuk:existence}{article}{
   author={Prasad, Gopal},
   author={Rapinchuk, Andrei S.},
   title={On the existence of isotropic forms of semi-simple algebraic
   groups over number fields with prescribed local behavior},
   journal={Adv. Math.},
   volume={207},
   date={2006},
   number={2},
   pages={646--660},
   issn={0001-8708},
   review={\MR{2271021}},
 }

 \bib{Serre:cohomologie-discrets}{article}{
   author={Serre, Jean-Pierre},
   title={Cohomologie des groupes discrets},
   language={French},
   conference={
      title={Prospects in mathematics},
      address={Proc. Sympos., Princeton Univ., Princeton, N.J.},
      date={1970},
   },
   book={
      series={Ann. of Math. Stud.},
      volume={No. 70},
      publisher={Princeton Univ. Press, Princeton, NJ},
   },
   date={1971},
   pages={77--169},
   review={\MR{0385006}},
 }
 
 \bib{Serre:galois-cohomology}{book}{
   author={Serre, Jean-Pierre},
   title={Galois cohomology},
   series={Springer Monographs in Mathematics},
   edition={English edition},
   note={Translated from the French by Patrick Ion and revised by the
   author},
   publisher={Springer-Verlag, Berlin},
   date={2002},
   pages={x+210},
   isbn={3-540-42192-0},
   review={\MR{1867431}},
 }

 \bib{Tits:classification}{article}{
   author={Tits, J.},
   title={Classification of algebraic semisimple groups},
   conference={
      title={Algebraic Groups and Discontinuous Subgroups},
      address={Proc. Sympos. Pure Math., Boulder, Colo.},
      date={1965},
   },
   book={
      publisher={Amer. Math. Soc., Providence, RI},
   },
   date={1966},
   pages={33--62},
   review={\MR{0224710}},
 }
 
\end{biblist}
\end{bibdiv}
\end{document}